\documentclass[12pt]{amsart}

\usepackage{amssymb}

\newtheorem{thm}{Theorem}[section]

\newtheorem{lem}[thm]{Lemma}
\newtheorem{cor}[thm]{Corollary}

\theoremstyle{definition}

\theoremstyle{remark}

\numberwithin{equation}{section}

\newcommand{\R}{\mathbf{R}}  
\newcommand{\N}{\mathbf{N}}
\newcommand{\C}{\mathbf{C}}
\newcommand{\Z}{\mathbf{Z}}

\newcommand{\Mod}[1]{\ (\textup{mod}\ #1)}

\providecommand{\id}{\operatorname{\mathbf{1}}}

\begin{document}


\title[The second moment of primitive $L$-functions]{A uniform asymptotic formula for the second moment of primitive $L$-functions on the critical line }


\author{Olga Balkanova}
\address{Institute for Applied Mathematics, Russian Academy of Sciences, Khabarovsk, Russia}
\email{olgabalkanova@gmail.com}

\author{Dmitry Frolenkov}
\address{Steklov Mathematical Institute of Russian Academy of Sciences, Moscow, Russia}
\email{frolenkov@mi.ras.ru}

\keywords{primitive forms, L-functions}
\subjclass[2010]{Primary: 11F12}

\begin{abstract}
We prove an asymptotic formula  for the second moment of automorphic $L$-functions of even weight  and prime power level. The error term is estimated uniformly in all parameters: level, weight, shift and twist.
\end{abstract}

\maketitle

\tableofcontents

\section{Introduction}

Many problems of different nature  can be phrased in terms of  central $L$-values. Therefore, asymptotic evaluation of moments of $L$-functions near the critical line is a major subject in analytic number theory.
In view of applications, one of the most important things is to estimate  error terms  optimally and uniformly with respect to different parameters.

In this paper, we consider the family of automorphic $L$-functions associated to primitive forms of  weight $2k\geq 2$ and level
$N=p^{\nu}$ with $p$ prime and $\nu \geq 2$.

Let $\gamma$ be the Euler constant, $\Lambda(n) $ be the von Mangoldt function, $\phi(n)$ be the Euler-totient function, $\mu(n)$ be the M\"{o}bius function
and $\psi(n)$ be the logarithmic derivative of the Gamma-function.
For complex $s$ and natural $n$ define
\begin{equation}
\tau_s(n)=\sum_{n_1n_2=n}\left( \frac{n_1}{n_2}\right)^s=
n^{-s}\sigma_{2s}(n).
\end{equation}
Introduce the identity operator
\begin{equation}
 \id_{c}=\begin{cases}
           1 & \text{ if c is true}\\
           0 & \text{otherwise}
         \end{cases}.
\end{equation}
Let $p$ be a prime and
\begin{equation}
\phi_{\nu}(N)=
\left\{
              \begin{array}{ll}
                1-p^{-1}, & \hbox{if $N=p^{\nu},\,\nu\geq3$;} \\
                1-(p-p^{-1})^{-1}, & \hbox{if $N=p^{\nu},\,\nu=2$.}
              \end{array}
\right.
\end{equation}

Our main result is the  asymptotic formula for the twisted harmonic second moment $M_2(l,0,it)$  defined by \eqref{def:moment}.
\begin{thm}\label{thm:main}
Let $k\geq 1 $, $t \in \R$, $T:=3+|t|$, $N=p^{\nu},$ $p$ is a prime and $\nu \geq 2$.
If $p|l$ then $M_2(l,0,it)=0$. Otherwise
\begin{multline}\label{eq:mt}
M_2(l,0,it)=\frac{\phi(N)\phi_{\nu}(N)}{N}\frac{\sigma_{-2it}(l)}{l^{1/2-it}}\times\\ \times \Biggl(
\log{N}+2\gamma-2\log{2\pi}+2\frac{\log{p}}{p-1}+\psi(k+it)+\psi(k-it)\Biggr)-\\
-\frac{\phi(N)\phi_{\nu}(N)}{N}\frac{1}{l^{1/2-it}}\sum_{d|l}\sum_{r|d}\mu(d/r)\tau_{it}(r)r^{-it}\log{r}+\\+
O_{\epsilon}\left( V_N(l)+\frac{1}{p}V_{N/p}(l)+\frac{1}{p^2}V_{N/p^2}(l)+
\frac{(lTkp)^{\epsilon}}{p^2\sqrt{l}}
\id_{N=p^2}
\right).
\end{multline}
Here
\begin{equation}\label{casesE}
V_N(l)\ll
\begin{cases}
(lkT^2)^{\epsilon}\frac{l^{1/2}(k+T)}{N}& l \geq N/(1+T^2);\\
\left(\frac{lT^2}{4(N-l)}\right)^k\frac{N^{\epsilon}}{l^{1/2}T}
&l<N/(1+T^2).
\end{cases}
\end{equation}

\end{thm}
The given asymptotic formula is uniform in all parameters: $N$, $p$, $l$, $k$ and $T$.
Fixing $k$ we obtain that
\begin{equation}
V_N(l)+\frac{1}{p}V_{N/p}(l)+\frac{1}{p^2}V_{N/p^2}(l)\ll_{k,\epsilon} \frac{l^{1/2}T}{N}(lNT)^{\epsilon}.
\end{equation}

This improves and generalizes several results in the prior literature.

First, the given paper extends methods of \cite{Byk, BF, BF2} to the case of primitive forms of prime power level.  We apply the technique of analytic continuation instead of approximate functional equation in order to prove the explicit formula for the  twisted second moment on the critical line.  This generalizes the formulas of Bykovskii \cite{Byk} , Iwaniec $\&$ Sarnak \cite{IS}  and Bykovskii $\&$ Frolenkov \cite{BF}. The most challenging case of weight  $2k=2$ is proved in \cite{BF} using the so-called "Hecke trick".

From the explicit formula we derive the asymptotics that generalizes theorem $1$ of \cite{Bui} proved for $N$ prime, $k=1$, $t=0$ and $l<N$.
Removing the restriction $l<N$ is of crucial importance here.
As shown in \cite{BalFrol} this is required to prove the best known lower bound on the proportion of non-vanishing $L$-values when $N=p^{\nu}$, $\nu$ is fixed and $p \rightarrow \infty$.

Furthermore, uniformity on shift $t$ allows establishing a positive proportion of non-vanishing at any point on the critical line $1/2+it$ such that $|t|<N$.

Finally, theorem \ref{thm:main} improves the result of Rouymi.
\begin{thm}\label{thm:R}(lemma 5 of \cite{R2})
Let  $k\geq 1$, $N=p^{\nu}$, $p$ is prime,  $\nu \geq 3$, $(l,p)=1$. Then for all $1\leq l \leq N$
\begin{multline}\label{m2rouymi}
M_2(l,0,0)=\frac{\tau(l)}{\sqrt{l}}\left( \frac{\phi(N)}{N}\right)^2\times \\ \times
\left( \log{\left( \frac{N}{4\pi^2l}\right)}+2 \left( \frac{\log{p}}{p-1}+\psi(k)+\gamma\right)\right)+O_{k,p}\left(\sqrt{l}\frac{(\log{N})^4}{\sqrt{N}}\right).
\end{multline}
\end{thm}

Note that if $t=0$, the main term of \eqref{eq:mt} coincides with the main term in \eqref{m2rouymi}. Indeed,
\begin{multline}
\sum_{d|l}\sum_{r|d}\mu(d/r)\tau_0(r)\log{r}
=\sum_{d|l}\sum_{r|d}\mu(r)\log{\frac{d}{r}}\sum_{a|\frac{d}{r}}1=\\=
\sum_{d|l}\sum_{a|d}\sum_{r|\frac{d}{a}}\mu(r)\log{\frac{d}{r}}
=\sum_{d|l}\left( \log{d}+\sum_{k|d}\Lambda(d/k)\right)=\\=
2\sum_{d|l}\log{d}=\sum_{d|l}\left(\log{d}+\log{\frac{l}{d}}\right)=\tau_0(l)\log{l}.
\end{multline}

One can also compare the error terms in theorems  \ref{thm:main} and \ref{thm:R} by setting $l=1$. Equation \eqref{m2rouymi} gives $O_{k,p}\left((\log{N})^4N^{-1/2}\right)$
and \eqref{casesE} implies that
\begin{equation}
V_N(1)+\frac{1}{p}V_{N/p}(1)+\frac{1}{p^2}V_{N/p^2}(1)\ll_{k,t,\epsilon} \frac{N^{\epsilon}}{p^2(N/p^2)^k}\ll_{k,t,\epsilon,p}N^{-k+\epsilon}.
\end{equation}

The paper is organized as follows. In sections \ref{section2} and \ref{section3} we recall some background information. Sections \ref{section4} and \ref{section5} are devoted to estimating special functions which appear as a part of error terms in the asymptotic formula for the second moment.
In section \ref{section6} we prove the exact formula  for the second moment when $N=p^{\nu}$, $\nu \geq 3$. The case $N=p^2$ is considered in section \ref{section7}. Finally, in the last section we derive the asymptotic formula by estimating uniformly in all parameters different types of error terms.

\section{Preliminary information}\label{section2}
Define
\begin{equation*}
\Gamma(u,v,\lambda;s)=\frac{\Gamma(\lambda-1/2+s/2)}{\Gamma(\lambda+1/2-s/2)}
\Gamma(1/2-u+v-s/2)\Gamma(1/2-u-v-s/2).
\end{equation*}
Using Stirling's formula for $\Re{v}=0$ we have
\begin{equation}
 \Gamma(u,v,\lambda;\sigma+it)\ll \frac{\exp(-\pi|t|/2)}{|t|^{1+2\Re{u}}}.
\end{equation}

The Gauss hypergeometric function is defined for $|z|<1$ by the power series
\begin{equation*}
{}_2F_1(a,b,c;z)=\frac{\Gamma(c)}{\Gamma(a)\Gamma(b)}\sum_{n=0}^{\infty}\frac{\Gamma(a+n)\Gamma(b+n)}{\Gamma(c+n)}\frac{z^n}{n!}.
\end{equation*}

The classical delta symbol is denoted by
\begin{equation}
\delta(a,b)=\begin{cases}
1, \quad a=b\\
0, \quad \text{otherwise}
\end{cases}.
\end{equation}

For $q \in \N$ and $a \in \Z$ let
\begin{equation}
\delta_q(a)=\begin{cases}
1, \quad a \equiv 0 \Mod{q}\\
0, \quad \text{otherwise}
\end{cases}.
\end{equation}

The Bessel function
\begin{equation}
J_s(z)=\sum_{n=0}^{\infty}\frac{(-1)^n}{\Gamma(n+1)\Gamma(n+1+s)}\left(\frac{z}{2}\right)^{s+2n}
\end{equation}
satisfies the Mellin-Barnes representation
\begin{equation}\label{Barnes}
J_{2\lambda-1}(y)=\frac{1}{4\pi i}\int\limits_{\Re s=\Delta}
\frac{\Gamma(\lambda-1/2+s/2)}{\Gamma(\lambda+1/2-s/2)}\left(\frac{y}{2}\right)^{-s}ds
\end{equation}
for $1-2\Re{\lambda}<\Delta<0$ and positive real  $y$.
The Lerch zeta function with parameters $\alpha, \beta \in \R$ is defined by
\begin{equation}
\zeta(\alpha,\beta;s)=\sum_{n+\alpha>0}\frac{\exp(2\pi i n\beta)}{(n+\alpha)^s}
\end{equation}
for $\Re{s}>1$.
This is a periodic function on $\beta$ with a period one; $\zeta(\alpha,\beta;s)$ can be holomorphically continued
on the whole complex plane except the point $s=1$ for $\beta \in \Z$, where it has a simple pole with residue $1$.

The Mellin transfom of $f:[0,\infty)\rightarrow \mathbb{C}$ is given by
\begin{gather}\label{Mellin.def}
\hat{f}(s)=\int_{0}^{\infty}f(x)x^{s-1}dx.
\end{gather}

The classical Kloosterman sum
\begin{equation}
 Kl(m,n;c)=\sum_{\substack{0 \leq x \leq c-1\\(x,c)=1}}
\exp{\left( 2\pi i\frac{mx+n\overline{x}}{c}\right)}, \quad
x\overline{x}\equiv 1 \Mod{c}
\end{equation}
satisfies the Weil bound
\begin{equation}
| Kl(m,n;c)|\leq \tau_0(c)\sqrt{(m,n,c)}\sqrt{c}.
\end{equation}
\begin{lem}(\cite{Royer}, lemma A.12)\label{Royer}
Let $m,n,c$ be three strictly positive integers and $p$ be a prime
number. Suppose $p^2|c$, $p|m$ and $p\not{|} n$. Then $Kl(m,n;c)=0$.
\end{lem}

\section{Primitive forms and the Petersson trace formula}\label{section3}

Let $S_{2k}(N)$ be the space of cusp forms of weight $2k \geq 2$
and level $N$. It is equipped with the Petersson inner product
\begin{equation} \label{Pet}
\langle f,g\rangle_N:=\int_{F_0(N)}f(z)\overline{g(z)}y^{k}\frac{dxdy}{y^2},
\end{equation}
where $F_0(N)$ is a fundamental domain of the action of the Hecke congruence
subgroup $\Gamma_0(N)$ on the upper-half plane $\mathbb{H}=\{z \in \C: \Im{z}>0\}$.

Any $f \in S_{2k}(N)$ has a Fourier expansion at infinity
\begin{equation}f(z)=\sum_{n\geq 1}a_f(n)e(nz).
\end{equation}

According to the Atkin-Lehner theory the space of cusp forms can be decomposed into two subspaces of new and old forms
\begin{equation}
S_{2k}(N)=S_{2k}^{new}(N)\oplus S_{2k}^{old}(N).
\end{equation}
We denote by $H_{2k}(N)$ an orthogonal basis of $S_{2k}(N)$ and by $H_{2k}^{*}(N)$ an orthogonal basis of $S_{2k}^{new}(N)$
consisting of primitive forms with normalized Fourier coefficients
\begin{equation}
\lambda_f(n):=a_f(n)n^{-(2k-1)/2},
\end{equation}
\begin{equation}
\lambda_f(1)=1.
\end{equation}

The coefficients $\lambda_f(n)$ satisfy the Hecke relation
\begin{equation}\label{multiplicity}
\lambda_f(m)\lambda_f(n)=\sum_{\substack{d|(m,n)\\(d,p)=1 }}\lambda_{f}\left( \frac{mn}{d^2}\right).
\end{equation}

Let $Re (s)>1$, then for $f \in
H_{2k}^{*}(N)$ we define an automorphic $L$-function
\begin{equation}
L(s,f)=\sum_{n \geq 1}\lambda_f(n)n^{-s}.
\end{equation}

The completed $L$-function
\begin{equation}
\Lambda(s,f)=\left(\frac{\sqrt{N}}{2 \pi}\right)^s\Gamma\left(s+\frac{2k-1}{2}\right)L(s,f)
\end{equation}
can be analytically continued on the whole complex plane and satisfies the functional equation
\begin{equation} \label{eq: functionalE}
\Lambda(s,f)=\epsilon_f\Lambda(1-s,f),
\end{equation}
where $s\in \C$ and $\epsilon_f=\pm1$.

We define the harmonic average as follows
\begin{equation} \label{harmonicAv}
\sum_{f \in H_{2k}(N)}^{h}:=\sum_{f \in
H_{2k}(N)}\frac{\Gamma(2k-1)}{(4\pi)^{2k-1} \langle f,f\rangle_N}.
\end{equation}

In this paper we study the twisted harmonic second moment of automorphic $L$-functions associated to primitive forms of  weight $2k\geq 2$ and level
$N=p^{\nu}$ with $p$ prime and $\nu \geq 2$
\begin{equation}\label{def:moment}
M_2(l,u,v)=\sum_{f \in H_{2k}^{*}(N)}^{h}\lambda_f(l)L_f(1/2+u+v)L_f(1/2+u-v).
\end{equation}

Our proof is based on the Petersson trace formula.
\begin{thm}\label{Petersson}
For $m,n \geq 1$ we have
\begin{multline}
\Delta_{2k,N}(m,n)=\sum_{f\in H_{2k} (N)}^{h}
\lambda_f(m)\lambda_f(n)=\\=\delta(m,n)+ 2\pi i^{-2k}\sum_{N | c}
\frac{Kl(m,n;c)}{c} J_{2k-1} ( \frac{4\pi\sqrt{mn}}{c} ).
\end{multline}
\end{thm}

More precisely, we apply the following generalization by Rouymi.
\begin{thm}\label{PetRou}(remark $4$ of \cite{R})
Let  $N=p^{\nu}$ with prime $p$ and $\nu \geq 2$. Then
\begin{multline}\label{eq:rtrace}
\Delta_{2k,N}^{*}(m,n):=\sum_{f\in H_{2k}^{*}(N)}^{h}
\lambda_f(m)\lambda_f(n)=\\=\begin{cases}
\Delta_{2k,N}(m,n)-\frac{\Delta_{2k,N/p}(m,n)}{p-p^{-1}}& \text{if $\nu=2$ and $(N,mn)=1$},\\
\Delta_{2k,N}(m,n)-\frac{\Delta_{2k,N/p}(m,n)}{p}& \text{if $\nu \geq 3$ and $(N,mn)=1$},\\
0& \text{if } (N,mn)=1.
\end{cases}
\end{multline}
 \end{thm}

\section{Integrals involving the Bessel functions}\label{section4}
Consider
\begin{equation}\label{Iz}
I(z)=\int_{0}^{\infty}J_{2k-1}(y)k^{+}(y\sqrt{z},1/2+it)dy,
\end{equation}
where
\begin{equation}\label{eq:k+}
k^{+}(y,v)=\frac{1}{2\cos{\pi v}}\left( J_{2v-1}(y)-J_{1-2v}(y) \right).
\end{equation}
\begin{lem}(\cite{BE2}, page 326) One has
\begin{equation}\label{MBJ}
J_{2k-1}(y)=\frac{1}{4\pi i}\int_{\Delta}\frac{\Gamma(k-1/2+s/2)}{\Gamma(k+1/2-s/2)}\left( \frac{y}{2}\right)^{-s}ds
\end{equation}
for some $1-2k<\Delta<0$.
\end{lem}
\begin{lem}\label{lem:jbes}(equation $10.9.8$ of \cite{HMF} )
For $|\Re{v}|<1$ and $x>0$
\begin{equation}
J_v(x)=\frac{2}{\pi}\int_{0}^{\infty}\sin{\left( x\cosh{y}-\frac{v\pi}{2} \right)}\cosh{(vy)}dy.
\end{equation}
\end{lem}
\begin{cor}
For $a>0$ and $v=it$ with $t \in \R$
\begin{multline}\label{k+a}
k^{+}(a,1/2+it)=\frac{2}{\pi}\int_{0}^{\infty}\cos{(a\cosh{z})}\cos{(2tz)}dz
=\\
=\frac{2}{\pi}\int_{1}^{\infty}\frac{\cos{(au)}}{\sqrt{u^2-1}}\cos{\left( 2t \log{(u+\sqrt{u^2-1})}\right)}du.
\end{multline}
\end{cor}
\begin{thm} \label{decompI}For $z<1$ the following decomposition takes place
\begin{equation}
I(z)=\frac{2}{\pi}\left( Q_0(z)+(-1)^kQ_1(z)\right),
\end{equation}
\begin{multline}\label{eq:Q_0}
Q_0(z)=\\ \int_{1}^{1/\sqrt{z}}\frac{ \cos{(2t\log{(y+\sqrt{y^2-1})})} }{\sqrt{1-zy^2}\sqrt{y^2-1}} \cos{\left[(2k-1)\arcsin{(y\sqrt{z})}\right]}dy,
\end{multline}
\begin{multline}\label{eq:Q_1}
Q_1(z)=\\ \int_{1/\sqrt{z}}^{\infty}
\frac{ \cos{(2t\log{(y+\sqrt{y^2-1})})} }{\sqrt{y^2z-1}\sqrt{y^2-1}}\frac{dy}{\left(y\sqrt{z}+\sqrt{y^2z-1}\right)^{2k-1}}.
\end{multline}
\end{thm}
\begin{cor}\label{boundI}
For $z<1$ one has
\begin{equation}
I(z)\ll 1+|\log{z}|+\frac{1}{(2k-1)\sqrt{1-z}}.
\end{equation}
\end{cor}
\begin{proof}
Estimating \eqref{eq:Q_0} we obtain
\begin{equation*}
Q_0(z)\ll \int_{1}^{1/\sqrt{z}}\frac{dy}{\sqrt{1-zy^2}\sqrt{y^2-1}}=\frac{1}{\sqrt{z}}\int_{1}^{1/\sqrt{z}}\frac{dy}{\sqrt{1/z-y^2}\sqrt{y^2-1}}.
\end{equation*}
Applying equation $3.152(10)$ of \cite{GR} with $a=\frac{1}{\sqrt{z}}$, $b=1$, $u=b$, $\lambda=\pi/2$
\begin{equation*}
\frac{1}{\sqrt{z}}\int_{1}^{1/\sqrt{z}}\frac{dy}{\sqrt{1/z-y^2}\sqrt{y^2-1}}=F\left( \frac{\pi}{2}\sqrt{1-z}\right),
\end{equation*}
where $F$ is the elliptic integral defined by $8.111(2)$ of \cite{GR} or by $19.2.4$ of \cite{HMF}.
Formulas $19.2.8$ and $19.9.2$ of \cite{HMF} give
\begin{equation*}
Q_0(z)\ll \log{\left( \frac{4}{\sqrt{z}}\right)}\left( 1+\frac{z}{4}\right)\ll \log{\frac{4}{\sqrt{z}}}.
\end{equation*}
Now consider \eqref{eq:Q_1}
\begin{equation*}
Q_1(z)\ll \int_{1/\sqrt{z}}^{\infty}\frac{dy}{\sqrt{y^2z-1}\sqrt{y^2-1}(y\sqrt{z}+\sqrt{y^2z-1})^{2k-1}}.
\end{equation*}
Setting $r:=y\sqrt{z}$ we have
\begin{multline*}
Q_1(z)\ll \frac{1}{\sqrt{z}}\int_{1}^{\infty}\frac{dr}{\sqrt{r^2-1}\sqrt{r^2/z-1}(r+\sqrt{r^2-1})^{2k-1}}\ll \\ \ll
\frac{1}{\sqrt{1-z}}\int_{1}^{\infty}\frac{dr}{\sqrt{r^2-1}(r+\sqrt{r^2-1})^{2k-1}}\ll \frac{1}{(2k-1)\sqrt{1-z}}.
\end{multline*}
The last two estimates and theorem \ref{decompI} yield the assertion.
\end{proof}
\subsection{Proof of theorem \ref{decompI}}
We are going to evaluate the integral
\begin{equation}
I(k,t,1/x)=\int_{0}^{\infty}J_{2k-1}(y)k^{+}\left(\frac{y}{\sqrt{x}},1/2+it\right)dy
\end{equation}
for $x>1$ using representation \eqref{k+a}.
However, in order to be able to change the order of integration in $I(k,t,1/x)$, the uniform convergence of all integrals is required.
To overcome this difficulty we make several preliminary steps.
First, let us fix a large real number $Z>0$ such that
\begin{multline}
k^{+}(a,1/2+it)=\\=\frac{2}{\pi}\int_{1}^{Z}\frac{\cos{(au)}}{\sqrt{u^2-1}}
 \cos{\left( 2t \log{(u+\sqrt{u^2-1})}\right)}du+O\left( \frac{1+|t|}{aZ}\right).
\end{multline}
Second, we introduce an extra parameter $b\geq 0$ to avoid the discontinuity of the $y$-integral and consider
\begin{equation}\label{eq:ibktx}
I_b(k,t,1/x)=\int_{0}^{\infty}J_{2k-1}(y)k^{+}\left(\frac{y}{\sqrt{x}},1/2+it\right)\exp{(-by)}dy.
\end{equation}
Note that
\begin{equation}
I(k,t,1/x)=\lim_{b \rightarrow 0}I_b(k,t,1/x).
\end{equation}
Finally, for some $\epsilon>0$ we split the integral over $u$ into three parts
\begin{equation}\label{eq:ka2}
k^{+}(a,1/2+it)=\mathfrak{I}_1(a)+\mathfrak{I}_2(a)+\mathfrak{I}_3(a)+O\left( \frac{1+|t|}{aZ}\right),
\end{equation}
where
\begin{equation}
\mathfrak{I}_1(a)=\frac{2}{\pi}\int_{1}^{\sqrt{x}(1-\epsilon)}\frac{\cos{(au)}}{\sqrt{u^2-1}}\cos{\left( 2t \log{(u+\sqrt{u^2-1})}\right)}du,
\end{equation}
\begin{equation}
\mathfrak{I}_2(a)=\frac{2}{\pi}\int_{\sqrt{x}(1-\epsilon)}^{\sqrt{x}(1+\epsilon)}\frac{\cos{(au)}}{\sqrt{u^2-1}} \cos{\left( 2t \log{(u+\sqrt{u^2-1})}\right)}du,
\end{equation}
\begin{equation}
\mathfrak{I}_3(a)=\frac{2}{\pi}\int_{\sqrt{x}(1+\epsilon)}^{Z}\frac{\cos{(au)}}{\sqrt{u^2-1}}\cos{\left( 2t \log{(u+\sqrt{u^2-1})}\right)}du.
\end{equation}
Next, we substitute equation \eqref{eq:ka2} with $a:=\frac{y}{\sqrt{x}}$ into $I_b(k,t,1/x)$. The contribution of the error term in \eqref{eq:ka2} into \eqref{eq:ibktx} is bounded by
\begin{equation*}
\int_{0}^{\infty}|J_{2k-1}(y)|\frac{(1+|t|)\sqrt{x}}{yZ}\exp{(-by)}dy\ll  \frac{(1+|t|)\sqrt{x}}{Z}.
\end{equation*}
\begin{lem} \label{lemL1} The integral
$\mathfrak{I}_{1}(y/\sqrt{x})$ contributes to $I(k,t,1/x)$ as
\begin{equation}
\frac{2}{\pi}Q_0(1/x)+O_{k}\left( \epsilon\sqrt{\frac{ x}{x-1}}\right).
\end{equation}
\end{lem}
\begin{proof}
Consider
\begin{multline*}\mathfrak{L}_1:=\frac{2}{\pi}
\int_{1}^{\sqrt{x}(1-\epsilon)}\frac{\cos{\left(2t\log{(u+\sqrt{u^2-1})} \right)}}{\sqrt{u^2-1}}\times\\ \times
\left(\int_{0}^{\infty}\exp{(-by)}J_{2k-1}(y)\cos{\left( \frac{u}{\sqrt{x}}y\right)}dy \right)du.
\end{multline*}
The inner integral
\begin{equation*}
IY_1:=\int_{0}^{\infty}\exp{(-by)}J_{2k-1}(y)\cos{\left( \frac{u}{\sqrt{x}}y\right)}dy
\end{equation*}
can be evaluated using Mellin-Barnes representation \eqref{MBJ}.
Changing the order of integration
\begin{multline*}
IY_1=\frac{1}{4\pi i}\int_{\Delta}2^s\frac{\Gamma(k-1/2+s/2)}{\Gamma(k+1/2-s/2)}\times\\ \times\left(\int_{0}^{\infty}\exp{(-by)}\cos{\left( \frac{u}{\sqrt{x}}y\right)}y^{-s}dy \right)ds.
\end{multline*}
For $\Re{s}<1$ it follows that (see \cite{F} pages $785-786$)
\begin{multline*}
\int_{0}^{\infty}\exp{(-by)}\cos{\left( \frac{u}{\sqrt{x}}y\right)}y^{-s}dy=\\=\frac{\Gamma(1-s)}{(u^2/x+b^2)^{(1-s)/2}}\cos{\left((1-s)\arctan{\frac{u}{\sqrt{x}b}} \right)}.
\end{multline*}
Therefore,
\begin{multline*}
IY_1=\frac{1}{4\pi i}\int_{\Delta}\frac{\Gamma(k-1/2+s/2)}{\Gamma(k+1/2-s/2)}\Gamma(1-s)\times\\ \times\cos{\left((1-s)\arctan{\frac{u}{\sqrt{x}b}} \right)}
\frac{2^s}{(u^2/x+b^2)^{(1-s)/2}}ds.
\end{multline*}
Since $b \rightarrow 0$ and $u/\sqrt{x}\leq 1-\epsilon$, we have $u^2/x+b^2\leq 1-\epsilon/2$.
Thus the expression under the integral is a rapidly decreasing function as $s \rightarrow \infty$. Moving the contour of integration to the right, we cross poles at $s_j=1+j$, $j=0,1,\ldots$ Hence
\begin{multline*}
IY_1=\frac{1}{2}\sum_{j=0}^{\infty}\frac{\Gamma(k+j/2)}{\Gamma(k-j/2)}\frac{(-1)^j}{j!}\cos{\left( j\arctan{\frac{u}{\sqrt{x}b}}\right)}\frac{2^{1+j}}{(u^2/x+b^2)^{-j/2}}.
\end{multline*}
Now we can substitute $IY_1$ into $\mathfrak{L}_1$ and compute the limit as $b \rightarrow 0$.
Uniform convergence on $b$ follows by Weierstrass M-Test  since $u^2/x+b^2\leq 1-\epsilon/2$.
Thus it is possible to interchange the order of summation and integration
\begin{multline*}
\mathfrak{L}_1=\frac{2}{\pi}\int_{1}^{\sqrt{x}(1-\epsilon)}\frac{\cos{\left( 2t \log{(u+\sqrt{u^2-1})}\right)}}{\sqrt{u^2-1}}\times \\ \times \sum_{j=0}^{\infty}
\frac{\Gamma(k+j/2)}{\Gamma(k-j/2)}\frac{(-1)^j\cos{(\pi j/2)}}{j!}\frac{2^j du}{(u/\sqrt{x})^{-j}}=\\=
\frac{2}{\pi}\int_{1}^{\sqrt{x}(1-\epsilon)}\frac{\cos{\left( 2t \log{(u+\sqrt{u^2-1})}\right)}}{\sqrt{u^2-1}}\times \\ \times
\sum_{m=0}^{k-1}
\frac{\Gamma(k+m)}{\Gamma(k-m)}\frac{(-1)^m}{(2m)!}\left(\frac{2u}{\sqrt{x}}\right)^{2m}du.
\end{multline*}
By formula $1.332(4)$ of \cite{GR}
\begin{equation*}
\frac{\cos{((2k-1)\arcsin{c})}}{\sqrt{1-c^2}}=(-1)^k\sum_{n=1}^{k}\frac{\Gamma(2k-n)}{\Gamma(n)(2k-2n)!}(-1)^{n}c^{2k-2n}.
\end{equation*}
Therefore,
\begin{equation*}
\mathfrak{L}_1=\frac{2}{\pi}Q_0(1/x)-E_1,
\end{equation*}
where $Q_0(x)$ is defined by \eqref{eq:Q_0} and
\begin{multline*}
E_1:=\frac{2}{\pi}\int_{\sqrt{x}(1-\epsilon)}^{\sqrt{x}}\frac{\cos{\left( 2t \log{(u+\sqrt{u^2-1})}\right)}}{\sqrt{u^2-1}}\times \\ \times \sum_{m=0}^{k-1}
\frac{\Gamma(k+m)}{\Gamma(k-m)}\frac{(-1)^m}{(2m)!}\left(\frac{2u}{\sqrt{x}}\right)^{2m}du.
\end{multline*}
The error term can be estimated as follows
\begin{equation*}
E_1\ll_{k} \frac{\epsilon \sqrt{x}}{\sqrt{x-1}}.
\end{equation*}
\end{proof}
\begin{lem}\label{lemL2}
The integral $\mathfrak{I}_{2}(y/\sqrt{x})$ contributes to $I(k,t,1/x)$ as
\begin{equation}
O\left( \epsilon^{1/3}\sqrt{\frac{x}{x-1}}+\epsilon\frac{x}{x-1}\left( t+\sqrt{\frac{x}{x-1}}\right)\right).
\end{equation}
\end{lem}
\begin{proof}
Integrating by parts we have
\begin{equation*}
\mathfrak{I}_2(a)=\mathfrak{I}_{2,1}(a)-\mathfrak{I}_{2,2}(a),
\end{equation*}
where
\begin{equation*}
\mathfrak{I}_{2,1}(a):= \frac{\sin{au}}{a}\frac{\cos{(2t\log{(u+\sqrt{u^2-1})})}}{\sqrt{u^2-1}}\Bigg|_{\sqrt{x}(1-\epsilon)}^{\sqrt{x}(1+\epsilon)},
\end{equation*}
\begin{equation*}
\mathfrak{I}_{2,2}(a):=\int_{\sqrt{x}(1-\epsilon)}^{\sqrt{x}(1+\epsilon)}\frac{\sin{au}}{a}\left( \frac{\cos{\left( 2t \log{(u+\sqrt{u^2-1})}\right)}}{\sqrt{u^2-1}}\right)^{'}du.
\end{equation*}
Using
\begin{equation*}
\left( \frac{\cos{\left( 2t \log{(u+\sqrt{u^2-1})}\right)}}{\sqrt{u^2-1}}\right)^{'}\ll \frac{t}{u^2-1}+\frac{u}{(u^2-1)^{3/2}},
\end{equation*}
we obtain
\begin{equation*}
\mathfrak{I}_{2,2}(a) \ll \frac{\epsilon\sqrt{x}}{a}\left( \frac{t}{x-1}+\frac{\sqrt{x}}{(x-1)^{3/2}}\right).
\end{equation*}
On the one hand, it is possible to estimate the non-integral term trivially
\begin{equation*}
\mathfrak{I}_{2,1}(a) \ll \frac{1}{a\sqrt{x-1}}.
\end{equation*}
On the other hand, we can apply
\begin{multline*}
\left(\frac{\sin{au}}{a}\frac{\cos{(2t\log{(u+\sqrt{u^2-1})})}}{\sqrt{u^2-1}}\right)^{'}\ll\\ \ll \frac{1}{(u^2-1)^{1/2}}+\frac{t}{a(u^2-1)}+\frac{u}{a(u^2-1)^{3/2}}.
\end{multline*}
Then the mean value theorem gives
\begin{equation*}
\mathfrak{I}_{2,1}(a) \ll \epsilon \sqrt{x}\left( \frac{1}{\sqrt{x-1}}+\frac{t}{a(x-1)}+\frac{\sqrt{x}}{a(x-1)^{3/2}}\right).
\end{equation*}
Combining all the results
\begin{equation*}
\mathfrak{I}_{2}(a) \ll \frac{\epsilon\sqrt{x}}{\sqrt{x-1}}\min{\left(1,\frac{1}{a\epsilon\sqrt{x}}\right)}+\frac{\epsilon\sqrt{x}}{a(x-1)}\left( t+\sqrt{\frac{x}{x-1}}\right).
\end{equation*}
Estimating
\begin{equation*}
\frac{\epsilon\sqrt{x}}{\sqrt{x-1}}\min{\left(1,\frac{1}{a\epsilon\sqrt{x}}\right)} \ll \frac{(\epsilon\sqrt{x})^{1/3}}{a^{2/3}\sqrt{x-1}}
\end{equation*}
and setting $a:=y/\sqrt{x}$ we have
\begin{equation*}
\mathfrak{I}_{2}(y/\sqrt{x})\ll\frac{\epsilon^{1/3}\sqrt{x}}{y^{2/3}\sqrt{x-1}}+\frac{\epsilon}{y}\frac{x}{x-1}\left(t+\sqrt{\frac{x}{x-1}}\right).
\end{equation*}
Therefore,  $\mathfrak{I}_{2}(y/\sqrt{x})$ contributes to $I(k,t,1/x)$ as
\begin{equation*}
O\left( \epsilon^{1/3}\sqrt{\frac{x}{x-1}}+\epsilon\frac{x}{x-1}\left( t+\sqrt{\frac{x}{x-1}}\right)\right).
\end{equation*}
\end{proof}

\begin{lem}\label{lemL3}
The integral $\mathfrak{I}_3(y/\sqrt{x})$ contributes to $I(k,t,1/x)$ as
\begin{equation}
\frac{2}{\pi}(-1)^kQ_1(1/x)+O\left(\epsilon^{1/2}\sqrt{\frac{x}{x-1}}+\frac{x^k}{Z^{2k}} \right).
\end{equation}
\end{lem}
\begin{proof}
We need to evaluate the following expression as $b \rightarrow 0$
\begin{multline*}
\mathfrak{L}_2:=\frac{2}{\pi}\int_{\sqrt{x}(1+\epsilon)}^{Z}\frac{\cos{\left( 2t\log{(u+\sqrt{u^2-1})}\right)}}{\sqrt{u^2-1}}\times\\
\times \left( \int_{0}^{\infty}\exp{(-by)}J_{2k-1}(y)\cos{\left( \frac{u}{\sqrt{x}}y\right)}dy\right)du.
\end{multline*}
Calculating the $y$-integral in the same way as in lemma \ref{lemL1} we obtain
\begin{multline*}
IY_2=\frac{1}{4\pi i}\int_{\Delta}\frac{\Gamma(k-1/2+s/2)}{\Gamma(k+1/2-s/2)}\Gamma(1-s)\times\\ \times\cos{\left((1-s)\arctan{\frac{u}{\sqrt{x}b}} \right)}
\frac{2^s}{(u^2/x+b^2)^{(1-s)/2}}ds.
\end{multline*}
Note that now $u^2/x+b^2>1$ and, therefore, we move the contour of integration to the left crossing poles at $s=1-2k-2j$, $j=0,1,\ldots$
This yields
\begin{multline*}
IY_2=\sum_{j=0}^{\infty}\frac{\Gamma(2k+2j)}{\Gamma(2k+j)}\frac{(-1)^j}{j!}\times \\ \times \cos{\left( (2k+2j)\arctan{\frac{u}{\sqrt{xb}}}\right)}\frac{2^{1-2k-2j}}{(u^2/x+b^2)^{k+j}}.
\end{multline*}
Substituting this into $\mathfrak{L}_2$ and letting $b\rightarrow 0$ gives
\begin{multline*}
\mathfrak{L}_2:=\frac{2}{\pi}\int_{\sqrt{x}(1+\epsilon)}^{Z}\frac{\cos{\left( 2t\log{(u+\sqrt{u^2-1})}\right)}}{\sqrt{u^2-1}}\times\\
\times \sum_{j=0}^{\infty}\frac{\Gamma(2k+2j)}{\Gamma(2k+j)}\frac{(-1)^k}{j!}\frac{2^{1-2k-2j}}{(u/\sqrt{x})^{2k+2j}}du.
\end{multline*}
Recall that
\begin{equation*}
\Gamma(2k+2j)=\frac{1}{\sqrt{\pi}}2^{2k+2j-1}\Gamma(k+j)\Gamma(k+j+1/2)
\end{equation*}
and
\begin{multline*}
\frac{1}{\sqrt{\pi}}\sum_{j=0}^{\infty}\frac{\Gamma(k+j)\Gamma(k+j+1/2)}{\Gamma(j+2k)j!}\frac{(-1)^k}{c^{2j+2k}}
=\frac{(-1)^k}{c^{2k}\sqrt{\pi}}\frac{\Gamma(k)\Gamma(k+1/2)}{\Gamma(2k)}\times\\ \times {}_2F_1(k,k+1/2,2k;1/c^2)=\frac{(-1)^k}{\sqrt{c^2-1}(c+\sqrt{c^2-1})^{2k-1}}.
\end{multline*}
The last equality follows by formula $2.8 (6)$ of \cite{BE}.
This yields
\begin{multline*}
\mathfrak{L}_2=\frac{2}{\pi}\int_{\sqrt{x}(1+\epsilon)}^{Z}\frac{\cos{\left( 2t\log{(u+\sqrt{u^2-1})}\right)}}{\sqrt{u^2-1}}\times\\
\times \frac{(-1)^kdu}{\sqrt{u^2/x-1}\left( u/\sqrt{x}+\sqrt{u^2/x-1}\right)^{2k-1}}.
\end{multline*}
Finally we split the integral into three parts
\begin{equation*}
\int_{\sqrt{x}(1+\epsilon)}^{Z}=\int_{\sqrt{x}(1+\epsilon)}^{\infty}-\int_{\sqrt{x}}^{\sqrt{x}(1+\epsilon)}-\int_{Z}^{\infty}.
\end{equation*}
The first part contributes as $\frac{2}{\pi}(-1)^kQ_1(1/x)$ and the last two
as $$O\left(\epsilon^{1/2}\sqrt{\frac{x}{x-1}}+\frac{x^k}{Z^{2k}} \right).$$
This completes the proof.
\end{proof}

Letting $Z\rightarrow \infty$, $\epsilon \rightarrow 0$  in lemmas \ref{lemL1}, \ref{lemL2} and \ref{lemL3}, we obtain the statement of theorem \ref{decompI}.

\section{Hypergeometric functions}\label{section5}
Let $\Re{\lambda}>1+\Re{u}\geq1$,  $\Re{v}=0$, $\Im{v}=t$ with $t \in \R$ and $T:=3+|t|$.
The main object of this section is the function defined via two integrals
\begin{equation}
H_{\lambda}(u,v,z)=\begin{cases}
1/(2\cos{\pi(\lambda-u)})z^{1/2-\lambda}I_1(1/z) & z>0\\
1/2 (-z)^{1/2-\lambda}I_2(-1/z) & z<0
\end{cases},
\end{equation}
where
\begin{gather*}
I_1(z)=\frac{1}{2\pi i}\int\limits_{\Re s=\Delta}\Gamma(u,v,\lambda;s)\sin\pi\left(u+\frac{s}{2}\right)z^{s/2}ds,\\
I_2(z)=\frac{1}{2\pi i}\int\limits_{\Re s=\Delta}\Gamma(u,v,\lambda;s)z^{s/2}ds
\end{gather*}
with $1-2\Re\lambda<\Delta<-1-2\Re u$.
\begin{lem}
For $z \leq 1$, $z\neq 0$, $\Re{u}\geq 0$
\begin{multline}
H_{\lambda}(u,v;z)=\\=\frac{\Gamma(\lambda-u+v)\Gamma(\lambda-u-v)}{\Gamma(2\lambda)}{}_2F_{1}(\lambda-u+v,\lambda-u-v,2\lambda;z).
\end{multline}
\end{lem}
\begin{proof}
This formula follows by moving the contour of integration to the left and computing the residues.
\end{proof}
\begin{lem} (\cite{BF3}) For $z>0$, $k \geq 1$
\begin{equation}\label{eqw:h2}
\cosh{(\pi t)} H_k(0,v,-z)\ll
\frac{T^{2k-1}}{2^{2k}\sqrt{k}}\left[1+O\left( z^{1/2}\frac{k+T}{\sqrt{k}}\right)\right].
\end{equation}
\end{lem}
\begin{lem}(\cite{BF2}, page $3$)
For $z>0$, $k \geq 1$
\begin{equation}\label{eqw:h5}
\cosh{(\pi t)}H_k(0,v,-z)\ll z^{-k+1/2}.
\end{equation}
\end{lem}

\begin{lem} For $z>0$, $k \geq 1$
\begin{multline}\label{eqw:h1}
\cosh{(\pi t)} H_k(0,v,-z)\ll
\frac{z^{-k}}{\sqrt{T}}\left[1+\log{(zk)}+ O\left( z^{-1/2}\frac{k+T}{\sqrt{T}}\right)\right].
\end{multline}
\end{lem}
\begin{proof}
The Mellin-Barnes represention( formula 15.6.6 of \cite{HMF}) gives
\begin{equation}
H_k(0,v,-z)=\frac{1}{2\pi i}\int_{\Re{s}=c}\frac{\Gamma(k+it+s)\Gamma(k-it+s)\Gamma(-s)}{\Gamma(2k+s)}x^sds
\end{equation}
for $-k<c<0$.
Moving the contour of integration to $-k-1/2$ we have
\begin{multline*}
H_k(0,v,-z)=z^{-k+it}\Gamma(2it)
\frac{\Gamma(k-it)}{\Gamma(k+it)}+\\+z^{-k-it}\Gamma(-2it)\frac{\Gamma(k+it)}{\Gamma(k-it)}+
O\left(z^{-k-1/2}J\right),
\end{multline*}
where
\begin{equation*}
J=\int_{-\infty}^{+\infty}\Biggl|
 \frac{\Gamma(-1/2+i(r+t))\Gamma(-1/2+i(r-t))\Gamma(k+1/2-ir)}{\Gamma(k-1/2+ir)}\Biggr|dr.
\end{equation*}
The last integral can be estimated using the standard identities for the Gamma function
\begin{equation*}
J \ll \int_{-\infty}^{+\infty}
\frac{\sqrt{k^2+r^2}(\cosh{2\pi r}+\cosh{2\pi t})^{-1/2}}{\sqrt{1/4+(r-t)^2}\sqrt{1/4+(r+t)^2}}dr \ll
\frac{k+t}{T}\frac{\log{T}}{\exp(\pi t)}.
\end{equation*}
Therefore, for $t \neq 0$
\begin{equation*}
\cosh{\pi t}H_k(0,v,-z)\ll \frac{z^{-k}}{\sqrt{T}}\left(1+O\left(\frac{z+t}{\sqrt{zT}}\right) \right).
\end{equation*}
Computing the limit as $t \rightarrow 0$
\begin{equation*}
H_k(0,0,-z)=z^{-k}(\log{z}-2\gamma-2\psi(k))+O(z^{-k-1/2}k).
\end{equation*}
The result follows from the last two estimates.
\end{proof}

\begin{lem}
For $0<z<1$, $k \geq 1$
\begin{equation}\label{eqw:h3}
z^kH_k(0,v,z)\ll \sqrt{\frac{z}{1-z}}\frac{1}{k+t}.
\end{equation}
\end{lem}
\begin{proof}
Using the relation
\begin{equation*}
{}_1F_2(a,b,c;z)=(1-z)^{-a}{}_1F_2(a,c-b,c;z/(z-1))
\end{equation*}
we have
\begin{multline*}
z^kH_k(0,v,z)=\frac{\Gamma(k-v)\Gamma(k+v)}{\Gamma(2k)}\frac{z^k}{(1-z)^{k-v}}\times\\
\times{}_1F_2\left(k-v,k-v,2k;-\frac{z}{1-z}\right)=\frac{\Gamma(k+v)}{\Gamma(k-v)}\frac{z^k}{(1-z)^{k-v}}\times\\ \times
\frac{1}{2\pi i}\int_{\Re{s}=c}\frac{\Gamma^2(k-v+s)\Gamma(-s)}{\Gamma(2k+s)}\left( \frac{z}{1-z}\right)^sds,
\end{multline*}
where
$-k<c<0$. Moving the contour of integration to $c=-k+1/2$ we obtain
\begin{equation*}
z^kH_k(0,v,z)\ll\sqrt{\frac{z}{1-z}}\int_{-\infty}^{+\infty}\frac{dr}{\cosh{\pi(r-t)}\sqrt{r^2+(k-1/2)^2}}.
\end{equation*}
Next, we estimate the last integral by splitting it into two parts
$$\int_{-\infty}^{+\infty}=\int_{-\infty}^{0}+\int_{0}^{+\infty}.$$ Let $t\geq0$ then
\begin{equation*}
\int_{-\infty}^{0}\frac{dr}{\cosh{\pi(r-t)}\sqrt{r^2+(k-1/2)^2}}\ll \frac{\exp{(-\pi t)}}{k}.
\end{equation*}
For $t<k-1/2$
\begin{equation*}
\int_{0}^{+\infty}\frac{dr}{\cosh{\pi(r-t)}\sqrt{r^2+(k-1/2)^2}}
\ll \frac{1}{k}
\end{equation*}
and for $t>k-1/2$
\begin{equation*}
\int_{0}^{+\infty}\frac{dr}{\cosh{\pi(r-t)}\sqrt{r^2+(k-1/2)^2}}\ll \frac{1}{t}.
\end{equation*}
This yields the result.
\end{proof}

\begin{lem}(\cite{BF2}, page $2$)
For $0<z\leq 1/2$, $k \geq 1$
\begin{equation}\label{eqw:h4}
z^kH_k(0,v,z)\ll \frac{z^k}{2^k}.
\end{equation}
\end{lem}

\begin{lem}\label{lem:h}
Let $\Re{v}=0$ and $\Re{\lambda}>\Re{u}>0$. For $0<z<1$
\begin{multline}\label{eq:Hk+}
2 \cos{\pi(\lambda-u)}H_{\lambda}(u,v,1/z)=\\=2\pi 2^{2u}z^{\lambda-u}\int_{0}^{\infty}J_{2\lambda-1}(x)k^{+}(x\sqrt{z},1/2+v)x^{-2u}dx.
\end{multline}
\end{lem}
\begin{proof}
By definition
\begin{equation*}
2 \cos{\pi(\lambda-u)}H_{\lambda}(u,v,1/z)=z^{\lambda-1/2}I_1(z).
\end{equation*}
Representation \eqref{MBJ} yields that the Mellin transform of $J$-Bessel function equals
\begin{equation*}
\widehat{J}_{2\lambda-1}(s)=2^{s-1}\frac{\Gamma(\lambda-1/2+s/2)}{\Gamma(\lambda+1/2-s/2)}, \quad 1-2\lambda<\Re{s}<3/2.
\end{equation*}
Let
\begin{equation*}
\gamma(1/2-u-s/2,1/2+v):=\frac{2^{-2u-s}}{\pi}\Gamma(1/2-u-s/2+v)\Gamma(1/2-u-s/2-v),
\end{equation*}
\begin{equation*}
g_1(x)=:x^{-2u}k^{+}(x\sqrt{z},1/2+v),
\end{equation*}
where $k^+$ is defined by \eqref{eq:k+}. Then
\begin{equation*}
z^{s/2}\gamma(1/2-u-s/2,1/2+v)\sin{\pi (u+s/2)}=z^{1/2-u}\widehat{g}_1(1-s).
\end{equation*}
Therefore,
\begin{equation*}
I_1(z)=z^{1/2-u}\frac{2\pi2^{2u}}{2\pi i}\int_{\Re{s}=\Delta}\widehat{J}_{2\lambda-1}(s)\widehat{g}_1(1-s)ds.
\end{equation*}
Note that for $1/2-2\Re{u}<\Delta<1-2u$
\begin{equation*}
\int_{0}^{\infty}|g_1(x)x^{-\Delta}|dx<\infty
\end{equation*}
and for $\Delta<0$ we have
\begin{equation*}
\int_{\Re{s}=\Delta}\Biggl| \widehat{J}_{2\lambda-1}(s) \Biggr|ds<\infty.
\end{equation*}
Therefore, for $\Re{u}>1/4$ and $\Re{v}=0$
\begin{equation*}
I_1(z)=2\pi 2^{2u}z^{1/2-u}\int_{0}^{\infty}g_1(x)J_{2\lambda-1}(x)dx.
\end{equation*}
Since
both parts of equation \eqref{eq:Hk+} are analytic functions in the larger
region $\Re{\lambda}>\Re{u}>0$
the result extends to this region.
\end{proof}


\section{The second moment: $\nu \geq 3$}\label{section6}
We prove theorem \ref{thm:main} for $k\geq 2$. The most delicate and technical case $k=1$ follows by combining our results with the methods developed in \cite{BF}. To keep the length of the paper reasonable, we omit the details and note that the case $k=1$ requires just minor modifications comparing to \cite{BF}.
\begin{lem}\label{lemM2}
Suppose that $\Re{u}>1/2$, $\Re{v}=0$, $k\geq 2$, $N=p^{\nu}$ and $\nu \geq 3$. Then
\begin{equation}
M_2(l,u,v)=S(l,u,v;N)-\frac{1}{p}S(l,u,v;N/p),
\end{equation}
where
\begin{multline}\label{eq:s}
S(l,u,v;N)=\\  \id_{(l,p)=1}\frac{1}{l^{1/2+u-v}}\sum_{d|l}d^{1/2+u-v}\sum_{\substack{m,n=1\\(mn,p)=1}}^{\infty}\frac{(m/n)^v}{(mn)^{1/2+u}}\Delta_{2k,N}\left( md,n\right).
\end{multline}
\end{lem}
\begin{proof}
It follows from definition that
\begin{equation*}
M_2(l,u,v)=\sum_{m,n=1}^{\infty}\frac{1}{(mn)^{1/2+u}}\left( \frac{m}{n}\right)^{v}\sum_{f \in H_{2k}^{*}(N)}^{h}\lambda_f(l)\lambda_f(m)\lambda_f(n).
\end{equation*}
Property of multiplicity \eqref{multiplicity} yields
\begin{multline*}
M_2(l,u,v)=\sum_{\substack{d|l\\ (d,p)=1}}\sum_{ m \equiv 0\Mod{d}}
\sum_{n =1}^{\infty}\frac{1}{(mn)^{1/2+u}}\left( \frac{m}{n}\right)^v\Delta^{*}_{2k,N}\left( \frac{ml}{d^2},n\right)=\\
=\sum_{\substack{d|l\\ (d,p)=1}}\frac{1}{d^{1/2+u-v}}\sum_{m,n=1}^{\infty}\frac{1}{(mn)^{1/2+u}}
\left( \frac{m}{n}\right)^v\Delta^{*}_{2k,N}\left( \frac{ml}{d},n\right).
\end{multline*}
According to theorem \ref{PetRou}
\begin{equation*}\Delta_{2k,N}^{*}\left( \frac{ml}{d},n\right)\neq  0\text{ only if }\left( \frac{ml}{d}n,p\right)=1.
\end{equation*}
Hence
$\left(\frac{l}{d},p\right)=1$.  Since $(d,p)=1$, we have $(l,p)=1$. Note that the condition $(d,p)=1$ in summation over $d$ is satisfied if $(l,p)=1$ because $d|l$. Therefore,
\begin{multline*}
M_2(l,u,v)=\sum_{d|l}\frac{\id_{(l,p)=1}}{d^{1/2+u-v}}\sum_{\substack{m,n=1\\(mn,p)=1}}^{\infty}\frac{1}{(mn)^{1/2+u}}\left(\frac{m}{n} \right)^{v}\Delta^{*}_{2k,N}\left( \frac{ml}{d},n\right)=\\=
\frac{\id_{(l,p)=1}}{l^{1/2+u-v}}\sum_{d|l}d^{1/2+u-v}\sum_{\substack{m,n=1\\(mn,p)=1}}^{\infty}\frac{1}{(mn)^{1/2+u}}\left(\frac{m}{n} \right)^{v}\Delta^{*}_{2k,N}\left( md,n\right).
\end{multline*}
Applying theorem \ref{PetRou} we obtain the assertion.
\end{proof}

Define
\begin{equation}\label{t1}
T_1(l,u,v;N)=\sum_{m,n=1}^{\infty}\frac{1}{(mn)^{1/2+u}}\left(\frac{m}{n} \right)^{v}\Delta_{2k,N}\left( md,n\right),
\end{equation}
\begin{equation}\label{t2}
T_2(l,u,v;N)=\sum_{m,n=1}^{\infty}\frac{1}{(pmn)^{1/2+u}}\left(\frac{pm}{n} \right)^{v}\Delta_{2k,N}\left( pmd,n\right),
\end{equation}
\begin{equation}\label{t3}
T_3(l,u,v;N)=\sum_{\substack{m,n=1\\ (m,p)=1}}^{\infty}\frac{1}{(pmn)^{1/2+u}}\left(\frac{m}{pn} \right)^{v}\Delta_{2k,N}\left( md,pn\right)
\end{equation}
and let for $i=1,2,3$
\begin{equation}\label{si}
S_i(l,u,v;N)=\frac{\id_{(l,p)=1}}{l^{1/2+u-v}}\sum_{d|l}d^{1/2+u-v}T_i(l,u,v;N).
\end{equation}
\begin{lem}\label{lemS=S1-S2-S3}
The following decomposition takes place
\begin{equation}
S(l,u,v;N)=S_1(l,u,v;N)-S_2(l,u,v;N)-S_3(l,u,v;N).
\end{equation}
\end{lem}
\begin{proof}
Note that
\begin{equation*}
\sum_{\substack{m,n=1\\(mn,p)=1}}^{\infty}f(m,n)=\sum_{m,n=1}^{\infty}f(m,n)-\sum_{m,n=1}^{\infty}f(pm,n)-\sum_{\substack{m,n=1\\ (m,p)=1}}^{\infty}f(m,pn).
\end{equation*}
Applying this equality to $S(l,u,v;N)$ yields the result.

\end{proof}

\begin{lem}\label{lemS3=0}
The sums $S_3(l,u,v;N)$ and $S_3(l,u,v;N/p)$ vanish.
\end{lem}

\begin{proof}
Consider
\begin{multline*}
\Delta_{2k,N}(dm,pn)=\delta(dm,pn)+\\+2\pi i^{-2k}\sum_{q \equiv \Mod{N}}\frac{Kl(dm,pn;q)}{q}J_{2k-1}\left( 4 \pi \frac{\sqrt{dmpn}}{q}\right).
\end{multline*}
Conditions $(d,p)=1$ and $(m,p)=1$ imply that $\delta(dm,pn)=0$.
Since $q \equiv 0 \Mod{p^2}$, lemma \ref{Royer} yields that $Kl(dm,pn;q)=0$ and, therefore, $T_3(l,u,v;N)=0$.
\end{proof}
We introduce
\begin{multline}\label{DN*}
D_N^{*}(u,v;\lambda)=
\sum_{m,n=1}^{\infty}\frac{1}{(mn)^{1/2+u}}\left(\frac{m}{n}\right)^v
\times\\ \times
\sum_{q\equiv 0\Mod{N}}\frac{Kl(dm,n;q)}{q}J_{2\lambda-1}\left(4\pi\frac{\sqrt{dmn}}{q}\right).
\end{multline}
Petersson's trace formula \ref{Petersson} gives the following representation for \eqref{t1}.
\begin{lem}
Let $\Re{u}>3/4$ and $\Re{v}=0$. Then
\begin{equation}\label{eq:T1}
T_1(l,u,v;N)=\frac{\zeta(1+2u)}{d^{1/2+u+v}}+2\pi i^{2k}D_N^{*}(u,v;\lambda).
\end{equation}
\end{lem}
For $\Re{s}>1+|\Re{v}|$ let
\begin{equation}
G^*(s,v;d,q)=\sum_{m,n=1}^{\infty}\frac{Kl(md,n;q)}{(mn)^s}\left( \frac{m}{n}\right)^v.
\end{equation}
This is a generalization of function $G(s,v;q)$ given on page $5$ of \cite{BF}, namely
$G^*(s,v;1,q)=G(s,v;q).$

\begin{lem}\label{lem:gstar}
For $q \in \N$ and $ v \in \C$ the function $G^*(s,v;d,q)$ can be analytically continued on the whole complex plane as a function of complex variable $s$. Furthermore, for $\Re{s}<-|\Re{v}|$  one has
\begin{multline}
G^*(s,v;d,q)=2\Gamma(1-s+v)\Gamma(1-s-v)\left(\frac{2\pi}{q}\right) ^{2s-2} \times\\
\times \left( -\cos{\pi s}\sum_{\substack{m,n=1\\(n,q)=1}}^{\infty}\frac{\delta_q(mn-d)}{(mn)^{1-s}}\left(\frac{m}{n} \right)^{-v} +\right.\\
\left.+\cos{\pi v}\sum_{\substack{m,n=1\\(n,q)=1}}^{\infty}\frac{\delta_q(mn+d)}{(mn)^{1-s}}\left(\frac{m}{n} \right)^{-v} \right).
\end{multline}
\end{lem}
\begin{proof}
This can be proved analogously to lemma $2.2$ of \cite{BF} using the identity below
\begin{equation*}
\sum_{\substack{a,b=0\\ ad\equiv m\Mod{q}\\b \equiv n \Mod{q}}}^{q-1}\delta_q(ab-1)=\sum_{\substack{a=0\\ ad \equiv m\Mod{q}}}^{q-1}\delta_q(an-1)=\delta_q(mn-d)\id_{(n,q)=1}.
\end{equation*}
\end{proof}
Applying lemma \ref{lem:gstar} and following the arguments of lemmas $3.1$-$3.4$ of \cite{BF} we obtain
\begin{lem}
Let
$\Re \lambda-1>\Re u>3/4$  and  $\Re v=0$. Then
\begin{multline}\label{eq:DN}
D_N^{*}(u,v;\lambda)=
\frac{(2\pi)^{2u-1}}{2\pi i}\int\limits_{\Re s=\Delta}\Gamma(u,v,\lambda;s)\times \\
\times\Biggl(
\sin\pi\left(u+\frac{s}{2}\right)\sum_{q\equiv 0\Mod{N}}\sum_{\substack{m,n=1\\ (n,q)=1}}^{\infty}
\frac{\delta_q(mn-d)}{q^{2u}(mn)^{1/2-u-s/2}}\left(\frac{m}{n}\right)^{-v}+\\
+\cos\pi v\sum_{q\equiv 0\Mod{N}}\sum_{\substack{m,n=1\\ (n,q)=1}}^{\infty}
\frac{\delta_q(mn+d)}{q^{2u}(mn)^{1/2-u-s/2}}\left(\frac{m}{n}\right)^{-v}
\Biggr)\frac{ds}{d^{s/2}},
\end{multline}
where
  $1-2\Re\lambda<\Delta<-1-2\Re u.$
\end{lem}

\begin{lem}
For
$\Re \lambda-1>\Re u>3/4$  and  $\Re v=0$ one has
\begin{multline}
D_N^{*}(u,v;\lambda)=
\frac{(2\pi)^{2u-1}}{2\pi i N^{2u}}\int\limits_{\Re s=\Delta}\Gamma(u,v,\lambda;s)\times \\ \times \Biggl(
\zeta(2u)\sin{\pi\left(u+\frac{s}{2}\right)}\sum_{r|d}\frac{\mu(d/r)}{d^{1/2+u-v-s/2}}\frac{\tau_v(r)}{r^{v-2u}}+\\
+\sin{\pi\left(u+\frac{s}{2}\right)}\sum_{r|d}\frac{\mu(d/r)}{(d/r)^{1/2+u-v-s/2}}
\sum_{\substack{n \geq (1-r)/N\\n \neq 0}}\frac{\tau_u(|n|)\tau_v(nN+r)}{|n|^{u}(nN+r)^{1/2-u-s/2}}+\\
+\cos{\pi v}\sum_{r|d}\frac{\mu(d/r)}{(d/r)^{1/2+u-v-s/2}}
\sum_{n \geq (1+r)/N}\frac{\tau_u(n)\tau_v(nN-r)}{n^{u}(nN-r)^{1/2-u-s/2}}
\Biggr)\frac{ds}{d^{s/2}},
\end{multline}
where
  $1-2\Re{\lambda}<\Delta<-1-2\Re{u}.$
\end{lem}
\begin{proof}
Consider
\begin{multline*}
P(z):=\sum_{q \equiv 0\Mod{N}}\sum_{\substack{m,n=1\\ (n,q)=1}}^{\infty}\frac{\delta_q(mn-d)}{q^{2u}(mn)^{z}}\left(\frac{m}{n} \right)^{-v}=\\
=\frac{1}{N^{2u}}\sum_{q=1}^{\infty}\sum_{\substack{m,n=1\\(n,qN)=1}}^{\infty}\frac{\delta_{qN}(mn-d)}{q^{2u}(mn)^z}\left( \frac{m}{n}\right)^{-v}.
\end{multline*}
Recall that  $(d,N)=1$. Hence the condition $mn \equiv d \Mod{qN}$ implies that $(n,N)=1$. Therefore, $(n,qN)=1$ can be replaced by $(n,q)=1$ in the sum over $n$. We remove the last coprimality condition by M\"{o}bius inversion
\begin{multline*}
P(z)=\frac{1}{N^{2u}}\sum_{m,n,q=1}^{\infty}\frac{\delta_{qN}(mn-d)}{q^{2u}(mn)^z}\left( \frac{m}{n}\right)^{-v}\sum_{k|(n,q)}\mu(k)=\\
=\frac{1}{N^{2u}}\sum_{k=1}^{\infty}\mu(k)\sum_{q\equiv 0 \Mod{k}}\frac{1}{q^{2u}}\sum_{\substack{m,n=1\\ n \equiv 0 \Mod{k}}}^{\infty}\frac{\delta_{qN}(mn-d)}{(mn)^z}\left( \frac{m}{n}\right)^{-v}=\\
=\frac{1}{N^{2u}}\sum_{k=1}^{\infty}\frac{\mu(k)}{k^{2u+z-v}}\sum_{q=1}^{\infty}\frac{1}{q^{2u}}\sum_{m,n=1}^{\infty}\frac{\delta_{qkN}(mnk-d)}{(mn)^z}\left( \frac{m}{n}\right)^{-v}.
\end{multline*}
It follows from $mnk\equiv d\Mod{qNk}$ that $k|d$. The change of variables $r:=\frac{d}{k}$ gives
\begin{equation*}
P(z)=\frac{1}{N^{2u}}\sum_{r|d}\frac{\mu(d/r)}{(d/r)^{2u+z-v}}\sum_{q=1}^{\infty}\frac{1}{q^{2u}}\sum_{m,n=1}^{\infty}\frac{\delta_{qN}(mn-r)}{(mn)^z}\left( \frac{m}{n}\right)^{-v}.
\end{equation*}
Let $mn:=r+aN$ with $a \in \Z$, then
\begin{multline*}
\sum_{q=1}^{\infty}\frac{1}{q^{2u}}\sum_{m,n=1}^{\infty}\frac{\delta_{qN}(mn-r)}{(mn)^z}\left(\frac{m}{n}\right)^{-v}=\\=
\zeta(2u)\frac{\tau_v(r)}{r^z}+\sum_{q=1}^{\infty}\sum_{\substack{a \geq (1-r)/N\\a\neq 0}}\frac{\delta_{qN}(aN)}{q^{2u}(r+aN)^z}\sum_{mn=r+aN}\left(\frac{m}{n} \right)^{-v}=\\
=\zeta(2u)\frac{\tau_v(r)}{r^z}+\sum_{\substack{a \geq (1-r)/N)\\ a \neq 0}}\frac{\tau_v(r+aN)\tau_u(|a|)}{(r+aN)^z|a|^{u}}.
\end{multline*}
Analogously,
\begin{equation*}
\sum_{q=1}^{\infty}\frac{1}{q^{2u}}\sum_{m,n=1}^{\infty}\frac{\delta_{qN}(mn+r)}{(mn)^z}\left( \frac{m}{n}\right)^{-v}=
\sum_{a\geq (1+r)/N}\frac{\tau_v(aN-r)\tau_u(a)}{(aN-r)^z a^u}.
\end{equation*}
Applying the last three formulas to \eqref{eq:DN}, we obtain the assertion.
\end{proof}
\begin{lem}
For $\Re{\lambda}>1,$ $\Re{u}=0$, $\Re{v}=0$ and $u \neq 0$
\begin{multline}\label{DN}
D_{N}^{*}(u,v;\lambda)=\frac{2(2\pi)^{2u-1}}{N^{2u}}\zeta(2u)\Gamma(2u)\cos{\pi (\lambda-u)}\times \\
\times \frac{\Gamma(\lambda-u+v)\Gamma(\lambda-u-v)}{\Gamma(\lambda+u+v)\Gamma(\lambda+u-v)}\sum_{r|d}\frac{\mu(d/r)}{d^{1/2+u-v}}\frac{\tau_v(r)}{r^{v-2u}}+\\+\frac{2(2\pi)^{2u-1}}{N^{2u}} \sum_{r|d}\frac{\mu(d/r)}{d^{1/2+u-v}r^{-\lambda-u+v}}\times \\ \times \Biggl(\cos{\pi(\lambda-u)}
\sum_{\substack{n \geq (1-r)/N\\ n \neq 0}}\frac{\tau_u(|n|)\tau_v(nN+r)}{|n|^u(nN+r)^{\lambda-u}}H_{\lambda}\left(u,v;\frac{r}{nN+r}\right)+\\
+\cos{\pi v}\sum_{n \geq (1+r)/N}\frac{\tau_u(n)\tau_v(nN-r)}{n^u(nN-r)^{\lambda-u}}H_{\lambda}\left(u,v;\frac{-r}{nN-r}\right)\Biggr).
\end{multline}

\end{lem}
\begin{proof}
The proof is analogous to lemma 3.6 of \cite{BF}.
\end{proof}

Let
\begin{equation}\label{eq:vn}
V_N(l)=\frac{2(-1)^k}{l^{1/2-v}}\sum_{d|l}\sum_{r|d}\mu(d/r)\left(V_{N,1}(r)+V_{N,2}(r)+V_{N,3}(r)\right),
\end{equation}
where
\begin{equation}\label{eq:v1}
V_{N,1}(r)= \frac{\cos{\pi v}}{r^{-k+v}}\sum_{n \geq (1+r)/N}\frac{\tau_0(n)\tau_v(nN-r)}{(nN-r)^k}H_k\left( 0,v;\frac{-r}{nN-r}\right),
\end{equation}
\begin{equation}\label{eq:v2}
V_{N,2}(r)= \frac{(-1)^k}{r^{-k+v}}\sum_{(1-r)/N \leq n\leq  -1}\frac{\tau_0(|n|)\tau_v(nN+r)}{(nN+r)^k}H_k\left( 0,v;\frac{r}{nN+r}\right),
\end{equation}
\begin{equation}\label{eq:v3}
V_{N,3}(r)=\frac{(-1)^k}{r^{-k+v}}\sum_{n\geq  1}\frac{\tau_0(n)\tau_v(nN+r)}{(nN+r)^k}H_k\left( 0,v;\frac{r}{nN+r}\right).
\end{equation}
\begin{lem}\label{lemS1S2}
Assume that $\Re{v}=0$ and $(l,p)=1$. Then
\begin{multline}\label{lims1}
S_1(l,0,v;N)=\frac{\sigma_{-2v}(l)}{l^{1/2-v}}\left(\log{N}+2(\gamma-\log{2\pi})+\psi(k+v)+\psi(k-v)\right)-\\-
\frac{1}{l^{1/2-v}}\sum_{d|l}\sum_{r|d}\mu\left( \frac{d}{r}\right)\tau_v(r)r^{-v}\log{r}+V_N(l)
\end{multline}
and
\begin{multline}\label{lims2}
S_2(l,0,v;N)=\frac{1}{p}\frac{\sigma_{-2v}(l)}{l^{1/2-v}}\times\\ \times \left(\log{(N/p)}+2\gamma-2\log{2\pi}+\psi(k+v)+\psi(k-v)\right)-\\-
\frac{1}{pl^{1/2-v}}\sum_{d|l}\sum_{r|d}\mu\left( \frac{d}{r}\right)\tau_v(r)r^{-v}\log{r}+\frac{1}{p}V_{N/p}(l).
\end{multline}
\end{lem}
\begin{proof}

We substitute \eqref{DN} into  \eqref{eq:T1} and continue $S_1(l,u,v;N)$  analytically  to the point $u=0$.
Note that the evaluation of $H_{\lambda}\left(u,v;z\right)$ for $z>1$ is not straightforward.
First, we use the representation of $H_{\lambda}\left(u,v;z\right)$ (for $\Re{u}>0$) given by lemma \ref{lem:h}.
 But it follows from theorem \ref{decompI} that the integral on the right-hand side of \eqref{eq:Hk+} converges at $u=0$ and
 thus we can eventually let $u=0$. This allows using \eqref{eq:Hk+} with $u=0$ in order to compute $H_{\lambda}\left(0,v;z\right)$ for $z>1.$

Since $k \geq 2$ all the series converge and, therefore, we can let $\lambda=k$.
In order to show that $S_1(l,u,v;N)$ doesn't have a pole at $u=0$ we consider its non-analytic summands
\begin{multline*}
\left( \frac{d}{l}\right)^{u}\Biggl(
\frac{\zeta(1+2u)}{d^{1/2+u+v}}+\left( \frac{4\pi^2}{N}\right)^{2u}\frac{\Gamma(k-u+v)\Gamma(k-u-v)}{\Gamma(k+u+v)\Gamma(k+u-v)}\times\\ \times\zeta(1-2u)
\sum_{r|d}\frac{\mu(d/r)}{d^{1/2+u-v}}\frac{\tau_v(r)}{r^{v-2u}}
\Biggr).
\end{multline*}
By M\"{o}bius inversion
\begin{equation*}
\sum_{r|d}\mu(r)\sigma_{-2v}(d/r)=d^{-2v}.
\end{equation*}
This implies that
\begin{equation*}
d^{-1/2-v}=\sum_{r|d}\frac{\mu(d/r)}{d^{1/2-v}}\frac{\tau_v(r)}{r^v}
\end{equation*}
and the singularity at $u=0$ gets canceled.
Computing the limit of $S_1(l,u,v;N)$ as $u \rightarrow 0$ gives \eqref{lims1}. The equality \eqref{lims2} can be established similarly.
\end{proof}

Using lemmas
\ref{lemM2}, \ref{lemS=S1-S2-S3}, \ref{lemS3=0} and \ref{lemS1S2} we obtain

\begin{thm}
For $\Re{v}=0$
\begin{multline}\label{m2}
M_2(l,0,v)=\left( \frac{\phi(N)}{N}\right)^{2}\frac{\sigma_{-2v}(l)}{l^{1/2-v}}\times\\ \times \Biggl(
\log{N}+2\gamma-2\log{2\pi}+2\frac{\log{p}}{p-1}+\psi(k+v)+\psi(k-v)\Biggr)-\\
-\left( \frac{\phi(N)}{N}\right)^{2}\frac{1}{l^{1/2-v}}\sum_{d|l}\sum_{r|d}\mu(d/r)\tau_v(r)r^{-v}\log{r}+\\+
O\left( V_N(l)+\frac{1}{p}V_{N/p}(l)+\frac{1}{p^2}V_{N/p^2}(l)\right).
\end{multline}
\end{thm}

\section{The second moment: $\nu=2$}\label{section7}
In this section we consider the most difficult case $N=p^2$.
\begin{lem}\label{lem:m2p2}
Suppose that $\Re{u}>1/2$, $\Re{v}=0$, $k \geq 2$ and $N=p^2$ with $p$ prime. Then
\begin{equation}
M_2(l,u,v)=S(l,u,v;p^2)-\frac{1}{p-p^{-1}}S(l,u,v;p),
\end{equation}
where $S(l,u,v;N)$ is defined by \eqref{eq:s}.
\end{lem}
Using lemmas \ref{lemS=S1-S2-S3}, \ref{lemS3=0} and \ref{lemS1S2} we obtain
that $S(l,0,v;p^2)$ satisfies the following asymptotic formula  with $N=p^2$
\begin{multline}\label{S(P2)}
S(l,0,v;p^2)=\frac{\phi(N)}{N}\frac{\sigma_{-2v}(l)}{l^{1/2-v}}\times\\ \times \Biggl(
\log{N}+2\gamma-2\log{2\pi}+\frac{\log{p}}{p-1}+\psi(k+v)+\psi(k-v)\Biggr)-\\
-\frac{\phi(N)}{N}\frac{1}{l^{1/2-v}}\sum_{d|l}\sum_{r|d}\mu(d/r)\tau_v(r)r^{-v}\log{r}+\\+
O\left( V_N(l)+\frac{1}{p}V_{N/p}(l)\right).
\end{multline}

The sum $S(l,u,v;p)$ is more involved because in that case we cannot apply the property of vanishing of Kloosterman sums given by lemma \ref{Royer}.

\begin{lem}
One has
\begin{equation}
S(l,u,v;p)=S_1-S_2-S_3+S_4,
\end{equation}
where
$S_i=S_i(l,u,v;p)$ satisfy \eqref{si}, $T_i(l,u,v;p)$ are defined by \eqref{t1}, \eqref{t2} for $i=1,2$
and
\begin{equation}
T_3(l,u,v;p)=\sum_{m,n=1}^{\infty}\frac{1}{(pmn)^{1/2+u}}\left(\frac{m}{pn} \right)^{v}\Delta_{2k,p}\left( md,pn\right),
\end{equation}
\begin{equation}
T_4(l,u,v;p)=\frac{1}{p^{1+2u}}\sum_{m,n=1}^{\infty}\frac{1}{(mn)^{1/2+u}}\left(\frac{m}{n} \right)^v\Delta_{2k,p}(mdp,np).
\end{equation}
\end{lem}
The asymptotic formula for $S_1(l,0,v;p)$ is given by equation \eqref{lims1}. In the following subsections we evaluate $S_2$, $S_3$ and $S_4$.
The main difference with the previous section is that for $N=p$ the asymptotic formulas for $S_i$, $i=2,3,4$ will contain extra summands coming from poles of the Lerch zeta function, namely
\begin{multline}\label{E2}
E_{p}^{2}(u,v,\lambda)=\frac{1}{p}\sum_{q|d}\frac{1}{q}\left(\frac{q}{2\pi \sqrt{d/p}} \right)^{1-2u+2v}\times\\ \times\sum_{\substack{b=1\\(b,qp)=1}}^{qp}\zeta(0,b/qp;1+2v)
\frac{\Gamma(\lambda-u+v)}{\Gamma(\lambda+u-v)},
\end{multline}
\begin{equation}\label{E3}
E_{p}^{3}(u,v,\lambda)=\frac{p^{2v}-1}{p}\frac{\zeta(1-2v)}{(2\pi \sqrt{d/p})^{1-2u-2v}}
\frac{\Gamma(\lambda-u-v)}{\Gamma(\lambda+u+v)},
\end{equation}
\begin{multline}\label{E4}
E_{p}^{4}(u,v,\lambda)=\frac{1}{p}\sum_{\substack{q|d\\q \neq 1}}\frac{1}{q}
\left(\frac{q}{2\pi \sqrt{d}} \right)^{1-2u+2v}\times \\ \times
\sum_{\substack{b=1\\(b,qp)=1}}^{qp}\zeta(0,b/q,1+2v)\frac{\Gamma(\lambda-u+v)}{\Gamma(\lambda+u-v)}+
\frac{p-1}{p}\times \\ \times \Biggl( \frac{\zeta(1+2v)}{(2\pi \sqrt{d})^{1-2u+2v}}\frac{\Gamma(\lambda-u+v)}{\Gamma(\lambda+u-v)}+
\frac{\zeta(1-2v)}{(2\pi \sqrt{d})^{1-2u-2v}}\frac{\Gamma(\lambda-u-v)}{\Gamma(\lambda+u+v)}\Biggr).
\end{multline}

The asymptotics of $S_2$ and $S_3$ contains also an error term of the following shape
\begin{multline}\label{def:wp}
W_p(l)=\frac{2i^{2k}}{l^{1/2-v}}\sum_{d|l}\sum_{r|d}\mu(d/r)r^{\lambda-v}\times \\
 \Biggl(
\cos{\pi\lambda}\sum_{a=1}^{\infty}\frac{\tau_v(a)\tau_0(|pa-r|)}{(ap)^{\lambda}}H_{\lambda}\left(0,v,\frac{r}{ap}\right)
+\\+
\cos{\pi v}\sum_{a=1}^{\infty}\frac{\tau_v(a)\tau_0(pa+r)}{(ap)^{\lambda}}H_{\lambda}\left(0,v,-\frac{r}{ap}\right)
\Biggr).
\end{multline}
Evaluating $S_2$, $S_3$ and $S_4$ we show that
\begin{multline*}
S(l,0,v,p)=-\frac{1-1/p}{l^{1/2-v}}\sum_{d|l}\sum_{r|d}\mu(d/r)\tau_{v}(r)r^{-v}\log{r}
+\frac{\sigma_{-2v}(l)}{l^{1/2-v}}\times \\ \times [ (1+1/p^2)\log{p}+(1-1/p)(2\gamma-2\log{2\pi}+\psi(k+v)+\psi(k-v))]+\\+
2\pi i^{2k}\sum_{d|l}\left( \frac{d}{l}\right)^{1/2-v}\left[
-\frac{E_{p}^2(0,v,\lambda)}{p^{1/2-v}}-
\frac{E^{3}_{p}(0,v,\lambda)}{p^{1/2+v}}+\frac{E_{p}^{4}(0,v,\lambda)}{p}\right]+\\+
\left(1+\frac{1}{p^2} \right)V_p(l)-\frac{p+1}{p^2}V_1(l)-\left(\frac{1}{p^{1+v}}+\frac{1}{p^{1-v}}\right)W_p(l).
\end{multline*}

Then lemma \ref{lem:m2p2} and equation \eqref{S(P2)} imply the main theorem.
\begin{thm} Let $\Re{v}=0$,  $k \geq 2$, $N=p^2$ with $p$ prime. If $p|l$ then $M_2(l,0,v)=0$. Otherwise
\begin{multline}\label{M2(p)}
M_2(l,0,v)=\frac{\sigma_{-2v}(l)}{l^{1/2-v}}\Biggl[2\left(1-\frac{p}{p^2-1}\right)\log{p}+\\ +\left( 1-\frac{1}{p}-\frac{1}{p+1}\right)(2\gamma
-2\log{2\pi}+\psi(k+v) +\psi(k-v))\Biggr]  -\\
-\frac{1-1/p-1/(p+1)}{l^{1/2-v}}\sum_{d|l}\sum_{r|d}\mu(d/r)\tau_{v}(r)r^{-v}\log{r}+\\+
\frac{2\pi i^{2k}}{p-p^{-1}}\sum_{d|l}\left( \frac{d}{l}\right)^{1/2-v}\left( -\frac{E_{p}^2(0,v,\lambda)}{p^{1/2-v}}-
\frac{E^{3}_{p}(0,v,\lambda)}{p^{1/2+v}}+\frac{E_{p}^{4}(0,v,l)}{p}\right)
+\\+O\left(V_{p^2}(l)+\frac{1}{p}V_{p}(l)+\frac{1}{p^2}V_{1}(l) +\frac{1}{p^2}W_{p}(l)\right).
\end{multline}
\end{thm}
\subsection{The sum $S_4$}
In this subsection we evaluate $S_4(u,v,\lambda;p)$ as $u \rightarrow 0$.
\begin{thm}\label{thm:s4}
For $\Re{v}=0$
\begin{multline*}
S_4(0,v,\lambda;p)=-\frac{1}{pl^{1/2-v}}\sum_{d|l}\sum_{r|d}\mu(d/r)\tau_{v}(r)r^{-v}\log{r}+\\+
\frac{\sigma_{-2v}(l)}{pl^{1/2-v}}\left(\frac{\log{p}}{p}+2\gamma -2\log{2\pi}+\psi(k+v)+\psi(k-v)\right)+\\+
\frac{p-1}{p^2}V_1(l)+\frac{1}{p^2}V_p(l)+\frac{2\pi i^{2k}}{p}\sum_{d|l}\left( \frac{d}{l}\right)^{1/2-v}E_{p}^{4}(0,v;\lambda).
\end{multline*}
\end{thm}

According to the trace formula \eqref{PetRou}
\begin{equation}
T_4(l,u,v;p)=\frac{1}{p^{1+2u}}\Biggl( \frac{\zeta(1+2u)}{d^{1/2+u}}+2\pi i^{2k}D_{p}^{4}(u,v;\lambda)
\Biggr),
\end{equation}
where
\begin{multline}\label{dp4}
D_{p}^{4}(u,v;\lambda)=\frac{1}{p}\sum_{m,n=1}^{\infty}\frac{\left(m/n \right)^{v}}{(mn)^{1/2+u}}
\times \\ \times\sum_{q=1}^{\infty}\frac{Kl(dpm,pn;qp)}{q}J_{2\lambda-1}\left(4\pi \frac{\sqrt{dmn}}{q}\right).
\end{multline}
For $\Re{s}>1+|\Re{v}|$ let
\begin{multline}
G^{*}_{4}(s,v;d,q)=\sum_{m,n=1}^{\infty}\frac{Kl(dpm,pn;qp)}{(mn)^{s}}\left(\frac{m}{n}\right)^{v}=\\
=\sum_{a,b=1}^{qp}\delta_{qp}(ab-1)\zeta(0,ad/q,s-v)\zeta(0,b/q,s+v).
\end{multline}
In particular, if $q|d$
\begin{equation}
G^{*}_{4}(s,v;d,q)=\zeta(s-v)\sum_{\substack{b=1\\(b,qp)=1}}^{qp} \zeta(0,b/q,s+v),
\end{equation}
\begin{equation}
G^{*}_{4}(s,v;d,1)=\phi(p)\zeta(s-v)\zeta(s+v).
\end{equation}
Note that $\zeta(0,ad/q,s-v)$ has a pole when $q|d$. Since $(b,qp)=1$ the function $\zeta(0,b/q,s+v)$ has a pole at $q=1$.

\begin{lem}\label{lem:gstar4}
For $q \in \N$ and $ v \in \C$ the function $G^{*}_{4}(s,v;d,q)$ can be meromorphically continued on the whole complex plane as a function of complex variable $s$. Furthermore, for $\Re{s}<-|\Re{v}|$  one has
\begin{multline}\label{g4*}
G^{*}_{4}(s,v;d,q)=2\Gamma(1-s+v)\Gamma(1-s-v)\left(\frac{2\pi}{q}\right) ^{2s-2} \times\\
\times (p-1+\delta_p(q))\left( -\cos{\pi s}\sum_{\substack{m,n=1\\(n,q)=1}}^{\infty}\frac{\delta_q(mn-d)}{(mn)^{1-s}}\left(\frac{m}{n} \right)^{-v}
+\right.\\
\left.+\cos{\pi v}\sum_{\substack{m,n=1\\(n,q)=1}}^{\infty}\frac{\delta_q(mn+d)}{(mn)^{1-s}}\left(\frac{m}{n} \right)^{-v} \right).
\end{multline}
\end{lem}
\begin{proof}
Following the arguments of lemma $2.2$ of \cite{BF} we obtain
\begin{multline*}
\sum_{a,b=1}^{qp}\delta_{qp}(ab-1)\zeta(ad/q,0,1-s+v)\zeta(b/q,0,1-s-v)=\\=q^{2-2s}\sum_{m,n=1}^{\infty}\frac{1}{(mn)^{1-s}}\left( \frac{n}{m}\right)^v
\sum_{\substack{a,b=1\\ b \equiv n\Mod{q}\\ad \equiv m\Mod{q}}}^{qp}\delta_{qp}(ab-1).
\end{multline*}
We proceed to evaluate the last sum.
There exists $j\Mod{p}$ such that $b=n+jq$. Since $ab \equiv 1 \mod{qp}$, we have $an+ajq \equiv1 \Mod{qp}$.
The requirements $(a,qp)=1$, $aj+\frac{an-1}{q}\equiv 0\Mod{p}$, $an\equiv 1\Mod{q}$ imply that $j$ is unique.
We are left to compute the number of solutions of the system
\begin{equation*}
\begin{cases}
(a,qp)=1\\
an \equiv 1\Mod{q}\\
m\equiv ad \Mod{q}
\end{cases}
\Leftrightarrow
\begin{cases}
(a,qp)=1\\
a\equiv n^{-1}\Mod{q}\\
(n,q)=1\\
m\equiv ad \Mod{q}
\end{cases}
\Leftrightarrow
\begin{cases}
(n,q)=1\\
mn\equiv d \Mod{q}\\
(n^{-1}+wq,qp)=1
\end{cases}
\end{equation*}
for $w\Mod{p}$.
 Since $(n^{-1},q)=1$, we have $(n^{-1}+wq,p)=1$.
The number of $w$ such that $n^{-1}+wq\equiv 0 \Mod{p}$  is equal to one if $(q,p)=1$ and is equal to zero otherwise.
Therefore,
\begin{equation*}
\sum_{\substack{a,b=1\\ b \equiv n\Mod{q}\\ad \equiv m\Mod{q}}}^{qp}\delta_{qp}(ab-1)=(p-1+\delta_p(q))\id_{(n,q)=1}\delta_q(mn-d).
\end{equation*}
The rest of the proof is analogous to lemma $2.2$ of \cite{BF}.

\end{proof}
Next, we apply  functional equation \eqref{g4*} to evaluate \eqref{dp4}.
\begin{lem}\label{lem:dp4}
Let $\Re{\lambda}-1>\Re{u}>3/4$ and $\Re{v}=0$. For $1-2\Re{\lambda}< \Delta< -1-2\Re{u}$ one has
\begin{multline}
D_{p}^{4}(u,v;\lambda)=\frac{(2\pi)^{2u-1}}{2\pi i}\int_{\Re{s}=\Delta}\Gamma(u,v,\lambda;s)\times\\
\times \Biggl( \sin{\pi(u+s/2)}\sum_{q=1}^{\infty}\sum_{\substack{m,n=1\\ (n,q)=1}}^{\infty}
\frac{\delta_{q}(mn-d)}{(mn)^{1/2-u-s/2}}\left( \frac{m}{n}\right)^{-v}\frac{p-1+\delta_p(q)}{pq^{2u}}+\\
+\cos{\pi v}\sum_{q=1}^{\infty}\sum_{\substack{m,n=1\\ (n,q)=1}}^{\infty}
\frac{\delta_{q}(mn+d)}{(mn)^{1/2-u-s/2}}\left( \frac{m}{n}\right)^{-v}\frac{p-1+\delta_p(q)}{pq^{2u}}\Biggr)\frac{ds}{d^{s/2}}
+E_{p}^{4}(u,v;\lambda),
\end{multline}
where the last summand, representing the contribution of poles, is defined by \eqref{E4}.
\end{lem}
Lemma \ref{lem:dp4} implies that
\begin{equation}\label{DP4uv}
D_{p}^{4}(u,v;\lambda)=\frac{p-1}{p}D_{1}^{*}(u,v;\lambda)+\frac{1}{p}D_{p}^{*}(u,v;\lambda)+E_{p}^{4}(u,v;\lambda),
\end{equation}
where the function $D_{N}^{*}(u,v;\lambda)$, defined  by \eqref{DN*}, satisfies \eqref{DN}.
Using \eqref{DP4uv}, \eqref{DN} and letting $u\rightarrow 0$, we compute the asymptotics of $S_4(u,v,\lambda;p)$.
Theorem \ref{thm:s4} follows.

\subsection{The sum $S_3$}
The main result of this subsection is the following theorem.
\begin{thm}\label{thm:s3}
For $\Re{v}=0$
\begin{multline*}
S_3(l,0,\lambda;p)=-\frac{1}{pl^{1/2-v}}\sum_{d|l}\sum_{r|d}\mu(d/r)\tau_{v}(r)r^{-v}\log{r}+\\+
\frac{\sigma_{-2v}(l)}{pl^{1/2-v}}\left(2\gamma -2\log{2\pi}+\psi(k+v)+\psi(k-v)\right)+\\
+\frac{1}{p}V_1(l)+\frac{1}{p^2}V_p(l)+\frac{1}{p^{1+v}}W_{p}(l)
+\frac{2\pi i^{2k}}{p}\sum_{d|l}\left( \frac{d}{l}\right)^{1/2-v}E_{p}^{3}(0,v;\lambda).
\end{multline*}
\end{thm}

According to the trace formula \eqref{PetRou}
\begin{multline}
S_3(l,u,v;p)=\sum_{d|l}\left( \frac{d}{l}\right)^{1/2-v}\frac{1}{p^{1/2+v}}\left(\frac{d}{pl}\right)^u\times \\ \times
\Biggl( \frac{\zeta(1+2u)}{p^{1/2+u-v}d^{1/2+u+v}}+2\pi i^{2k}D_{p}^{3}(u,v;\lambda)
\Biggr),
\end{multline}
where
\begin{multline}\label{dp3}
D_{p}^{3}(u,v;\lambda)=\frac{1}{p}\sum_{m,n=1}^{\infty}\frac{\left(m/n \right)^{v}}{(mn)^{1/2+u}}
\times \\ \times\sum_{q=1}^{\infty}\frac{Kl(dm,pn;qp)}{q}J_{2\lambda-1}\left(4\pi \frac{\sqrt{dmn/p}}{q}\right).
\end{multline}
For $\Re{s}>1+|\Re{v}|$ let
\begin{multline}
G^{*}_{3}(s,v;d,q)=\sum_{m,n=1}^{\infty}\frac{Kl(dm,pn;qp)}{(mn)^{s}}\left(\frac{m}{n}\right)^{v}=\\
=\sum_{a,b=1}^{qp}\delta_{qp}(ab-1)\zeta(0,ad/qp,s-v)\zeta(0,b/q,s+v).
\end{multline}
Since $(d,p)=1$, $(a,qp)=1$, the function $\zeta(0,ad/qp,s-v)$ has no poles. The second function
$\zeta(0,b/q,s+v)$ has a pole if $q=1$.
Namely,
\begin{equation}
G^{*}_{3}(s,v;d,1)=(p^{1-s+v}-1)\zeta(s-v)\zeta(s+v)
\end{equation}
has a pole at $s=1+v$.
\begin{lem}\label{lem:gstar3}
For $q \in \N$ and $ v \in \C$ the function $G^{*}_{3}(s,v;d,q)$ can be meromorphically continued on the whole complex plane as a function of complex variable $s$. Furthermore, for $\Re{s}<-|\Re{v}|$  one has
\begin{multline}\label{g3*}
G^{*}_{3}(s,v;d,q)=2\Gamma(1-s+v)\Gamma(1-s-v)\left(\frac{2\pi}{q}\right) ^{2s-2}  p^{1-s+v}\times\\
\times\Biggl( -\cos{\pi s}\sum_{\substack{m,n=1\\(n,q)=1\\(m,p)=1}}^{\infty}\frac{\delta_q(mn-d)}{(mn)^{1-s}}\left(\frac{m}{n} \right)^{-v}
+\\+\cos{\pi v}\sum_{\substack{m,n=1\\(n,q)=1\\(m,p)=1}}^{\infty}\frac{\delta_q(mn+d)}{(mn)^{1-s}}\left(\frac{m}{n} \right)^{-v} \Biggr).
\end{multline}
\end{lem}
\begin{proof}
Following the proof of lemma $2.2$ of \cite{BF} we need to calculate the sum
\begin{multline*}
\sum_{q,b=1}^{qp}\delta_{qp}(ab-1)\zeta(ad/(qp),0,1-s+v)\zeta(b/q,0,1-s-v)=\\
=q^{2-2s}p^{1-s+v}\sum_{m,n=1}^{\infty}\frac{(n/m)^{v}}{(mn)^{1-s}}
\sum_{\substack{a,b=1\\ad\equiv m\Mod{qp}\\b \equiv n \Mod{q}}}^{qp}
\delta_{qp}(ab-1).
\end{multline*}
The last sum can be evaluated by computing the number of solution $a,b$ of the following system
\begin{equation*}
\begin{cases}
ab\equiv 1 \Mod{qp}\\
ad \equiv m \Mod{qp}\\
b=n+wq, \quad w \Mod{p}
\end{cases}
\Leftrightarrow
\begin{cases}
an+awq\equiv 1\Mod{qp}\\
(a,p)=1\\
ad \equiv m \Mod{qp}
\end{cases}.
\end{equation*}
Since $aw+\frac{an-1}{q}\equiv 0 \Mod{p}$, the value of $w$ is unique. Hence
\begin{equation*}
\begin{cases}
an \equiv 1 \Mod{q}\\
ad\equiv m \Mod{qp}\\
(a,p)=1
\end{cases}
\Leftrightarrow
\begin{cases}
(a,p)=1\\
(n,q)=1\\
ad \equiv m \Mod{qp}\\
a=n^{-1}+jq, \quad j\Mod{p}
\end{cases}.
\end{equation*}
This gives $(n,q)=1$, $(n^{-1}+jq,p)=1$, $dn^{-1}+djq\equiv m \Mod{qp}$. Therefore,
$dj+\frac{dn^{-1}-m}{q}\equiv 0 \Mod{p}$. Since $(d,p)=1$, we have $j=\frac{m-dn^{-1}}{q}d^{*}\Mod{p}$, where $nn^{-1}\equiv 1\Mod{q}$,
$dd^{*}\equiv 1\Mod{p}$.
Next, we compute the number of $j_0=\frac{m-dn^{-1}}{q}d^{*}\Mod{p}$ for which the condition $(n^{-1}+j_0q,p)=1$ is satisfied.
If $n^{-1}+j_0q\equiv 0\Mod{p}$, then $n^{-1}+(m-dn^{-1})d^{*}\equiv 0\Mod{p}$. Thus $md^{*}\equiv 0\Mod{p}$ and, since
$(d,p)=1$, we have $m \equiv 0\Mod{p}$. Hence there is no solutions if $m \equiv 0\Mod{p}$ and there is one solution $j_0$ if
$(m,p)=1$. The final set of conditions is the following: $(n,q)=1$, $(m,p)=1$, $mn \equiv d\Mod{q}$. This implies the statement
 of the lemma.

\end{proof}

Next, we apply functional equation \eqref{g3*} to evaluate \eqref{dp3}.
\begin{lem}
Let $\Re{\lambda}-1>\Re{u}>3/4$ and $\Re{v}=0$. For $1-2\Re{\lambda}< \Delta< -1-2\Re{u}$ one has
\begin{multline}
D_{p}^{3}(u,v;\lambda)=E_{p}^{3}(u,v;\lambda)+\frac{(2\pi)^{2u-1}}{2\pi i}\int_{\Re{s}=\Delta}\Gamma(u,v,\lambda;s)
\times \\ \times \Biggl(
\sin{\pi(u+s/2)}
\sum_{q=1}^{\infty}\sum_{\substack{m,n=1\\ (m,p)=1\\(n,q)=1}}^{\infty}
\frac{\delta_q(mn-d)}{q^{2u}(mn)^{1/2-u-s/2}}\left(\frac{m}{n}\right)^{-v}+\\+\cos{\pi v}
\sum_{q=1}^{\infty}\sum_{\substack{m,n=1\\ (m,p)=1\\(n,q)=1}}^{\infty}
\frac{\delta_q(mn+d)}{q^{2u}(mn)^{1/2-u-s/2}}\left(\frac{m}{n}\right)^{-v} \Biggr)\frac{p^{1/2-u-s/2+v}ds}{p(\sqrt{d/p})^s}.
\end{multline}
\end{lem}

\begin{lem}\label{lem:dp3l}
Let $\Re{\lambda}-1>\Re{u}>3/4$ and $\Re{v}=0$. For $1-2\Re{\lambda}< \Delta< -1-2\Re{u}$ one has
\begin{multline}\label{eq:dp3m}
D_{p}^{3}(u,v;\lambda)=\frac{1}{p^{1/2+u-v}}D^{*}_{1}(u,v,\lambda)+E_{p}^{3}(u,v;\lambda)+\\+
\frac{(2\pi)^{2u-1}}{2\pi i}\int_{\Re{s}=\Delta}\Gamma(u,v,\lambda;s)\times\\ \times \Biggl( \sin{\pi(u+s/2)}
\sum_{r|d}\frac{\mu(d/r)}{(d/r)^{1/2+u-v-s/2}}\sum_{a=1}^{\infty}\frac{\tau_v(a)\tau_u(|pa-r|)}{a^{1/2-u-s/2}|pa-r|^{u}}+\\+
\cos{\pi v}\sum_{r|d}\frac{\mu(d/r)}{(d/r)^{1/2+u-v-s/2}}\sum_{a=1}^{\infty}\frac{\tau_v(a)\tau_u(pa+r)}{a^{1/2-u-s/2}
(pa+r)^{u}}
\Biggr)\frac{ds}{p(d/p)^{s/2}}.
\end{multline}
\end{lem}
\begin{proof}
Consider the following decomposition
\begin{equation*}
\sum_{q=1}^{\infty}\sum_{\substack{m,n=1\\(m,p)=(n,q)=1}}^{\infty}\frac{\delta_q(mn-d)}{q^{2u}(mn)^z}\left( \frac{m}{n}\right)^{-v}:=
A_1-A_2,
\end{equation*}
where
\begin{equation*}
A_1=\sum_{q=1}^{\infty}\sum_{\substack{m,n=1\\(n,q)=1}}^{\infty}\frac{\delta_q(mn-d)}{q^{2u}(mn)^z}
\left( \frac{m}{n}\right)^{-v},
\end{equation*}
\begin{equation*}
A_2=\frac{1}{p^{z+v}}\sum_{q=1}^{\infty}\sum_{\substack{m,n=1\\(n,q)=1}}^{\infty}
\frac{\delta_q(pmn-d)}{q^{2u}(mn)^z}\left( \frac{m}{n}\right)^{-v}.
\end{equation*}
The sum $A_1$ gives the first term in \eqref{eq:dp3m}.
By M\"{o}bius inversion
\begin{equation*}
A_2=\frac{1}{p^{z+v}}\sum_{r|d}\frac{\mu(d/r)}{(d/r)^{z+2u-v}}\sum_{q=1}^{\infty}\frac{1}{q^{2u}}\sum_{m,n=1}^{\infty}
\frac{\delta_q(pmn-r)}{(mn)^z}\left( \frac{m}{n}\right)^{-v}.
\end{equation*}
Since $r|d$ we have $(r,p)=1$ and, therefore, there is no $m,n$ such that $pmn=r$.
Thus
\begin{multline*}
\sum_{q=1}^{\infty}\frac{1}{q^{2u}}\sum_{m,n=1}^{\infty}
\frac{\delta_q(pmn-r)}{(mn)^z}\left( \frac{m}{n}\right)^{-v}=
\sum_{q=1}^{\infty}\frac{1}{q^{2u}}\sum_{a=1}^{\infty}
\frac{\delta_q(pa-r)}{a^z}\tau_v(a)=\\=
\sum_{a=1}^{\infty}\frac{\tau_v(a)}{a^z}\sigma_{-2u}(|pa-z|)=\sum_{a=1}^{\infty}\frac{\tau_v(a)\tau_u(|pa-r|)}{a^z|pa-r|^u}
\end{multline*}
and
\begin{equation*}
A_2=\frac{1}{p^{z+v}}\sum_{r|d}\frac{\mu(d/r)}{(d/r)^{z+2u-v}}\sum_{a=1}^{\infty}\frac{\tau_v(a)\tau_u(|pa-r|)}{a^z|pa-r|^u}.
\end{equation*}

\end{proof}

Integrating over $s$ we obtain
\begin{multline}\label{DP3uv}
D_{p}^{3}(u,v;\lambda)=\frac{1}{p^{1/2+u-v}}D_{1}^{*}(u,v;\lambda)+E_{p}^{3}(u,v;\lambda)+\\+
\frac{2(2\pi)^{2u-1}}{p^{\lambda+1/2}}\sum_{r|d}\frac{\mu(d/r)r^{\lambda+u-v}}{d^{1/2+u-v}}\times \\ \times \Biggl(
\cos{\pi(\lambda-u)}\sum_{a=1}^{\infty}\frac{\tau_v(a)\tau_u(|pa-r|)}{a^{\lambda-u}|pa-r|^u}H_{\lambda}\left(u,v,\frac{r}{ap}\right)+\\+
\cos{\pi v}\sum_{a=1}^{\infty}\frac{\tau_v(a)\tau_u(pa+r)}{a^{\lambda-u}(pa+r)^u}H_{\lambda}\left(u,v,-\frac{r}{ap}\right)
\Biggr),
\end{multline}
where the function $D_{1}^{*}(u,v;\lambda)$, defined  by \eqref{DN*}, satisfies \eqref{DN}.
Letting $u\rightarrow 0$, we prove theorem \ref{thm:s3}.

\subsection{The sum $S_2$}
This subsection is devoted to proving the formula below.
\begin{thm}\label{thm:s2}
For $\Re{v}=0$
\begin{multline}
S_2(l,0,\lambda;p)=-\frac{1}{pl^{1/2-v}}\sum_{d|l}\sum_{r|d}\mu(d/r)\tau_{v}(r)r^{-v}\log{r}+\\+
\frac{\sigma_{-2v}(l)}{pl^{1/2-v}}\left(2\gamma -2\log{2\pi}+\psi(k+v)+\psi(k-v)\right)+\\+
\frac{1}{p}V_1(l)+\frac{1}{p^{1-v}}W_p(l)+\frac{2\pi i^{2k}}{p^{1/2-v}}
\sum_{d|l}\left( \frac{d}{l}\right)^{1/2-v}E_{p}^{2}(0,v,\lambda).
\end{multline}
\end{thm}

Applying the trace formula \eqref{PetRou} we have
\begin{multline}
S_2(l,u,v;p)=\sum_{d|l}\left( \frac{d}{l}\right)^{1/2-v}\frac{1}{p^{1/2-v}}\left(\frac{d}{pl}\right)^u\times \\ \times
\Biggl( \frac{\zeta(1+2u)}{(dp)^{1/2+u+v}}+2\pi i^{2k}D_{p}^{2}(u,v;\lambda)
\Biggr),
\end{multline}
where
\begin{multline}\label{dp2}
D_{p}^{2}(u,v;\lambda)=\frac{1}{p}\sum_{m,n=1}^{\infty}\frac{\left(m/n \right)^{v}}{(mn)^{1/2+u}}
\times \\ \times\sum_{q=1}^{\infty}\frac{Kl(dmp,n;qp)}{q}J_{2\lambda-1}\left(4\pi \frac{\sqrt{dmn/p}}{q}\right).
\end{multline}
For $\Re{s}>1+|\Re{v}|$ let
\begin{multline}
G^{*}_{2}(s,v;d,q)=\sum_{m,n=1}^{\infty}\frac{Kl(dpm,n;qp)}{(mn)^{s}}\left(\frac{m}{n}\right)^{v}=\\
=\sum_{a,b=1}^{qp}\delta_{q}(ab-1)\zeta(0,ad/q,s-v)\zeta(0,b/(qp),s+v).
\end{multline}
Note that $\zeta(0, ad/q;s-v)$ has a pole when $q|d$ and $\zeta(0,b/(qp),s+v)$ does not have poles.
\begin{lem}\label{lem:gstar2}
For $q \in \N$ and $ v \in \C$ the function $G^{*}_{2}(s,v;d,q)$ can be meromorphically continued on the whole complex plane as a function of complex variable $s$. Furthermore, for $\Re{s}<-|\Re{v}|$  one has
\begin{multline}\label{g2*}
G^{*}_{2}(s,v;d,q)=2\Gamma(1-s+v)\Gamma(1-s-v)\left(\frac{2\pi}{q}\right) ^{2s-2}  p^{1-s-v}\times\\
\times\Biggl( -\cos{\pi s}\sum_{\substack{m,n=1\\(n,qp)=1}}^{\infty}\frac{\delta_q(mn-d)}{(mn)^{1-s}}\left(\frac{m}{n} \right)^{-v}
+\\+\cos{\pi v}\sum_{\substack{m,n=1\\(n,qp)=1}}^{\infty}\frac{\delta_q(mn+d)}{(mn)^{1-s}}\left(\frac{m}{n} \right)^{-v} \Biggr).
\end{multline}
\end{lem}
\begin{proof}
Our arguments follow the proof of lemma $2.2$ of \cite{BF} with the difference that now we need to compute
\begin{multline*}
\sum_{a,b=1}^{qp}\delta_q(ab-1)\zeta(ad/q,0,1-s+v)\zeta(b/(qp),0,1-s-v)=\\=
q^{2-2s}p^{1-s-v}\sum_{m,n=1}^{\infty}\frac{(n/m)^v}{(mn)^{1-s}}
\sum_{\substack{a,b=1\\ ad \equiv m\Mod{q}\\b\equiv n \Mod{qp} }}^{qp}\delta_{qp}(ab-1).
\end{multline*}
To evaluate the sum over $a,b$ we consider the following system
\begin{equation*}
\begin{cases}
ab \equiv 1\Mod{qp}\\
ad \equiv m \Mod{q}\\
b \equiv n\Mod{qp}
\end{cases}
\Leftrightarrow
\begin{cases}
an\equiv 1\Mod{qp}\\
ad\equiv m\Mod{q}
\end{cases}
\Leftrightarrow
\end{equation*}
\begin{equation*}
\Leftrightarrow
\begin{cases}
an \equiv 1\Mod{qp}\\
adp\equiv mp\Mod{qp}
\end{cases}
\Leftrightarrow
\begin{cases}
(n,qp)=1\\
dp\equiv mnp\Mod{qp}
\end{cases}.
\end{equation*}
This gives the condition $(n,qp)=1$, $mn\equiv d\Mod{q}$ in the final formula.
\end{proof}

Next, we apply functional equation \eqref{g2*} to evaluate \eqref{dp2}.
\begin{lem}
Let $\Re{\lambda}-1>\Re{u}>3/4$ and $\Re{v}=0$. For $1-2\Re{\lambda}< \Delta< -1-2\Re{u}$ one has
\begin{multline}
D_{p}^{2}(u,v;\lambda)=E_{p}^{2}(u,v;\lambda)+\frac{(2\pi)^{2u-1}}{2\pi i}\int_{\Re{s}=\Delta}\Gamma(u,v,\lambda;s)
\times \\ \times \Biggl(
\sin{\pi(u+s/2)}
\sum_{q=1}^{\infty}\sum_{\substack{m,n=1\\ (n,qp)=1}}^{\infty}
\frac{\delta_q(mn-d)}{q^{2u}(mn)^{1/2-u-s/2}}\left(\frac{m}{n}\right)^{-v}+\\+\cos{\pi v}
\sum_{q=1}^{\infty}\sum_{\substack{m,n=1\\ (n,qp)=1}}^{\infty}
\frac{\delta_q(mn+d)}{q^{2u}(mn)^{1/2-u-s/2}}\left(\frac{m}{n}\right)^{-v} \Biggr)
\frac{p^{1/2-u-s/2-v}ds}{p(\sqrt{d/p})^s}.
\end{multline}
\end{lem}
\begin{lem}
Let $\Re{\lambda}-1>\Re{u}>3/4$ and $\Re{v}=0$. For $1-2\Re{\lambda}< \Delta< -1-2\Re{u}$ one has
\begin{multline}
D_{p}^{2}(u,v;\lambda)=\frac{1}{p^{1/2+u+v}}D^{*}_{1}(u,v,\lambda)+E_{p}^{2}(u,v;\lambda)+\\+
\frac{(2\pi)^{2u-1}}{2\pi i}\int_{\Re{s}=\Delta}\Gamma(u,v,\lambda;s)\times\\ \times \Biggl( \sin{\pi(u+s/2)}
\sum_{r|d}\frac{\mu(d/r)}{(d/r)^{1/2+u-v-s/2}}\sum_{a=1}^{\infty}\frac{\tau_v(a)\tau_u(|pa-r|)}{a^{1/2-u-s/2}|pa-r|^{u}}+\\+
\cos{\pi v}\sum_{r|d}\frac{\mu(d/r)}{(d/r)^{1/2+u-v-s/2}}\sum_{a=1}^{\infty}\frac{\tau_v(a)\tau_u(pa+r)}{a^{1/2-u-s/2}
(pa+r)^{u}}
\Biggr)\frac{ds}{p(d/p)^{s/2}}.
\end{multline}
\end{lem}
\begin{proof}
Consider
\begin{multline*}
\sum_{q=1}^{\infty}\sum_{\substack{m,n=1\\(n,qp)=1}}^{\infty}\frac{\delta_q(mn-d)}{q^{2u}(mn)^z}\left(\frac{m}{n} \right)^{-v}
=\sum_{q=1}^{\infty}\sum_{\substack{m,n=1\\(n,q)=1}}^{\infty}\frac{\delta_q(mn-d)}{q^{2u}(mn)^z}\left(\frac{m}{n} \right)^{-v}
-\\-\frac{1}{p^{z-v}}\sum_{q=1}^{\infty}\sum_{\substack{m,n=1\\(pn,q)=1}}^{\infty}
\frac{\delta_q(pmn-d)}{q^{2u}(mn)^z}\left(\frac{m}{n} \right)^{-v}
.
\end{multline*}
Since $pmn\equiv d\Mod{q}$ and $(d,p)=1$, one has $(q,p)=1$ and the condition $(pn,q)=1$
in the second sum can be replaced by $(n,q)=1$. The rest of the proof is similar to lemma \ref{lem:dp3l}.
\end{proof}

Integrating over $s$ and letting $u\rightarrow 0$ we prove theorem \ref{thm:s2}.

\section{Error terms}\label{section8}
In this section we estimate the error terms defined by \eqref{eq:vn}, \eqref{eq:v1}, \eqref{eq:v2}, \eqref{eq:v3},  \eqref{E2}, \eqref{E3}, \eqref{E4} and \eqref{def:wp}.
Let
\begin{equation*}
v:=it, \quad t \in \R, \quad T:=3+|t|, \quad \lambda:=k,
\end{equation*}
\begin{equation*}
n_0:=\left[\frac{r+1}{N}\right]+1,\quad
n_1:=5n_0,\quad
n_2:=\frac{r}{N}(1+T^2).
\end{equation*}
\begin{lem}\label{lem:N1}
One has
\begin{equation}\label{V1}
V_{N,1}(r)\ll\begin{cases}
\frac{r}{N}(rkT^2)^{\epsilon}(k+T) &  r\geq  N/(1+T^2);\\
\left( \frac{rT^2}{4(N-r)}\right)^k\frac{N^{\epsilon}}{T}  &r< N/(1+T^2).
\end{cases}
\end{equation}
\end{lem}
\begin{proof}
Estimating  \eqref{eq:v1} we obtain
\begin{equation}\label{eq:v1.1}
V_{N,1}(r)\ll\cosh{\pi t}\sum_{n \geq n_0}\frac{n^{\epsilon}(nN-r)^{\epsilon}r^k}{(nN-r)^k}H_k\left( 0,v;\frac{-r}{nN-r}\right).
\end{equation}

\begin{enumerate}
\item
Assume that $r\geq N/2$. We split the sum over $n$ in \eqref{eq:v1.1} into three parts:
 $n_0\leq n < n_1, \, n_1\leq n< n_2$ and $n \geq n_2$.
Estimate \eqref{eqw:h1} implies that the first sum is bounded by
\begin{multline*}
\sum_{n_0\leq n<n_1}
(nN)^{\epsilon}\frac{1}{\sqrt{T}}\left[(rk)^{\epsilon}+\frac{k+T}{\sqrt{T}}\sqrt{\frac{nN-r}{r}}\right]\ll\\
\ll \frac{(rT)^{\epsilon}}{\sqrt{T}}\left[ n_1k^{\epsilon}+\frac{k+T}{\sqrt{T}}\frac{r}{N}\sqrt{\frac{n_1N-r}{r}}\right]\ll (rkT)^{\epsilon}\frac{r}{N}\frac{k+T}{T}.
\end{multline*}
Inequality \eqref{eqw:h5} gives the following bound for the second sum
\begin{multline*}
\sum_{n_1\leq n<n_2}
(nN)^{\epsilon}\frac{\sqrt{r}}{\sqrt{nN-r}}
\ll
(rT)^{\epsilon}r^{1/2}\times \\ \times
\left(\frac{1}{(n_1N-r)^{1/2}}+\frac{(n_2N-r)^{1/2}}{N}\right)\ll(Tr)^{\epsilon}\frac{rT}{N}.
\end{multline*}
Applying \eqref{eqw:h2} we estimate the third sum
\begin{multline}\label{n>n2}
\sum_{n \geq n_2}(nN-r)^{\epsilon}\frac{r^k}{(nN-r)^k}\frac{T^{2k-1}}{4^k\sqrt{k}}\left[ 1+\frac{k+T}{\sqrt{k}}\sqrt{\frac{r}{nN-r}}\right]\ll\\
\ll \frac{r^kT^{2k-1}}{4^k\sqrt{k}}\left[\frac{(n_2N-r)^{\epsilon}}{N(n_2N-r)^{k-1}}+
\frac{k+r}{\sqrt{k}}\frac{\sqrt{r}(n_2N-r)^{\epsilon}}{N(n_2N-r)^{k-1/2}} \right]\ll \\ \ll (rT^2)^{\epsilon} \frac{rT}{4^kN}.
\end{multline}
Combining the last three bounds we have
\begin{equation*}
V_{N,1}(r)\ll \frac{r}{N}(rkT^2)^{\epsilon}(k+T) \quad\text{for } r\geq N/2.
\end{equation*}

  \item
Assume that $N/(1+T^2)\leq r<N/2$. We split the sum over $n$ in \eqref{eq:v1.1} into two parts:
 $n_0\leq n < n_2$ and $n \geq n_2$.
Estimate \eqref{eqw:h5} implies that the first sum is bounded by
\begin{multline*}
\sum_{n_0\leq n<n_2}
(nN)^{\epsilon}\frac{\sqrt{r}}{\sqrt{nN-r}}
\ll
(rT)^{\epsilon}r^{1/2}\times \\ \times
\left(\frac{1}{(n_0N-r)^{1/2}}+\frac{(n_2N-r)^{1/2}}{N}\right)\ll(Tr)^{\epsilon}\frac{rT}{N}.
\end{multline*}
It follows from the last estimate and \eqref{n>n2} that
\begin{equation*}
V_{N,1}(r)\ll \frac{r}{N}(rkT^2)^{\epsilon}(k+T) \quad\text{for } N/(1+T^2)\leq r<N/2.
\end{equation*}

\item
Assume that  $r <N/(1+T^2)$. Applying \eqref{eqw:h2} we obtain (in the same way as in \eqref{n>n2})
\begin{equation*}
V_{N,1}(r)\ll \left( \frac{rT^2}{4(N-r)}\right)^k\frac{N^{\epsilon}}{T}\quad\text{for }r <N/(1+T^2) .
\end{equation*}
\end{enumerate}
\end{proof}

\begin{lem}\label{lem:N2}For any $\epsilon>0$, $r>N$
\begin{equation}
V_{N,2}(r)\ll_{\epsilon}\frac{r^{1+\epsilon}}{N}.
\end{equation}
\end{lem}
\begin{proof}
By lemma \ref{lem:h} we have that
\begin{multline*}
V_{N,2}(r)\ll \sum_{(1-r)/N\leq n\leq -1}\frac{\tau_{0}(|n|)\tau_{v}(nN+r)}{r^{v}}\int_{0}^{\infty}J_{2k-1}(x)\times \\ \times k^{+}\left(x\sqrt{\frac{nN+r}{r}},1/2+v\right)dx
\ll \sum_{1\leq n \leq (r-1)/N}\bigg{|}I\left( \frac{r-nN}{r}\right)\bigg{|},
\end{multline*}
where $I(z)$ is defined by \eqref{Iz}.
Then corollary \ref{boundI} gives the desired result.
\end{proof}

\begin{lem}\label{lem:N3}
For any $\epsilon>0$
\begin{equation}
V_{N,3}(r)\ll\begin{cases}
\frac{(rT^2)^{\epsilon}r}{N} &  r\geq  N/(1+T^2);\\
\left( \frac{r}{2(N+r)}\right)^kN^{\epsilon} &r< N/(1+T^2).
\end{cases}
\end{equation}
\end{lem}
\begin{proof}
The first bound follows from \eqref{eqw:h3} and the second one from \eqref{eqw:h4}.
\end{proof}

\begin{thm} For any $\epsilon>0$
\begin{equation}\label{VN(l)bound}
V_N(l)\ll
\begin{cases}
(lkT^2)^{\epsilon}\frac{l^{1/2}(k+T)}{N}& l \geq N/(1+T^2);\\
\left(\frac{lT^2}{4(N-l)}\right)^k\frac{N^{\epsilon}}{l^{1/2}T}
&l<N/(1+T^2).
\end{cases}
\end{equation}
\end{thm}
\begin{proof}
By lemmas \ref{lem:N1}, \ref{lem:N2} and \ref{lem:N3}
\begin{equation*}
V_N(l)\ll \frac{1}{l^{1/2}}\sum_{r|l}\bigg(V_{N,1}(r)+V_{N,2}(r)+V_{N,3}(r)\bigg)\ll
\frac{1}{l^{1/2}}\sum_{r|l}V_{N,1}(r).
\end{equation*}
If $l<N/(1+T^2)$ one can apply  the second bound in \eqref{V1} to obtain \eqref{VN(l)bound}.
Assume
$l \geq N/(1+T^2)$. Then for $N^{1-10\epsilon}/(1+T^2)<r<N/(1+T^2)$ we have
\begin{equation*}
\left( \frac{rT^2}{4(N-r)}\right)^k\frac{N^{\epsilon}}{T}\ll
\frac{r}{N}(rkT^2)^{\epsilon}(k+T).
\end{equation*}
 Next, we split the sum over $r$ into two parts. The assertion follows by applying estimates \eqref{V1}, i.e.
 we use the first bound for $r>N^{1-10\epsilon}/(1+T^2)$  and the second bound
for $r<N^{1-10\epsilon}/(1+T^2)$.

\end{proof}

\begin{lem}
One has
\begin{equation}
\frac{1}{p^2}W_p(l)\ll \frac{1}{p}V_p(l).
\end{equation}
\end{lem}
\begin{proof}
Consider
\begin{equation*}
W_p(l)=\frac{2(-1)^k}{l^{1/2-v}}\sum_{d|l}\sum_{r|d}\mu(d/r)\left(W_{p,1}(r)+W_{p,2}(r)+W_{p,3}(r)\right),
\end{equation*}
where
\begin{equation*}
W_{p,1}(r)= \frac{\cos{\pi v}}{r^{-k+v}}\sum_{a=1}^{\infty}
\frac{\tau_v(a)\tau_0(pa+r)}{(ap)^k}H_k\left( 0,v;\frac{-r}{ap}\right),
\end{equation*}
\begin{equation*}
W_{p,2}(r)= \frac{(-1)^k}{r^{-k+v}}\sum_{a\geq r/p}\frac{\tau_{v}(a)\tau_0(pa-r)}{(ap)^k}H_k\left( 0,v;\frac{r}{ap}\right),
\end{equation*}
\begin{equation*}
W_{p,3}(r)=\frac{(-1)^k}{r^{-k+v}}\sum_{a<r/p}\frac{\tau_{v}(a)\tau_0(|pa-r|)}{(ap)^k}H_k\left( 0,v;\frac{r}{ap}\right).
\end{equation*}
Analogously to lemma \ref{lem:N1} we obtain
\begin{equation*}
W_{p,1}(r)\ll
\begin{cases}
\frac{r}{p}(rkT^2)^{\epsilon}(k+T) ,\quad r \geq p/T^2\\
\left(\frac{rT^2}{4p}\right)^k\frac{1}{T\sqrt{k}}, \quad r < p/T^2.
\end{cases}
\end{equation*}
Therefore, $W_{p,1}(r)\ll V_{p,1}(r)$.
Next, we estimate  $W_{p,2}(r)$. It is convenient to split the summation over $a$ into two parts
\begin{equation*}
\sum_{a \geq r/p}=\sum_{a\geq 2r/p}+\sum_{r/p\leq a < 2r/p}.
\end{equation*}
According to \eqref{eqw:h4} the first sum can be bounded as follows
\begin{multline*}
\frac{(-1)^k}{r^{-k+v}}\sum_{a\geq 2r/p}\frac{\tau_{v}(a)\tau_0(pa-r)}{(ap)^k}H_k\left( 0,v;\frac{r}{ap}\right)\ll\\
\ll\sum_{a\geq 2r/p}(ap)^{\epsilon}\frac{r^k}{(2ap)^k}\ll
\begin{cases}
(r/(2p))^kp^{\epsilon}, \quad r<p/2\\
rp^{\epsilon}/(p2^k), \quad r>p/2.
\end{cases}
\end{multline*}
Therefore, the contribution of this term doesn't exceed $V_{p,3}(r).$ If $r>p/2$
we estimate the second sum using \eqref{eqw:h3} and letting $r=r_0+sp, \quad 0\leq r_0\leq p-1$. Since
$r|l$, $(l,p)=1$ one has $r_0 \neq 0$. It follows from $a \geq r/p$ that $a\geq s+1$, and, therefore,
$ap-r \geq p-r_0$. Thus
\begin{multline*}
\frac{(-1)^k}{r^{-k+v}}\sum_{r/p\leq a < 2r/p}\frac{\tau_{v}(a)\tau_0(pa-r)}{(ap)^k}H_k\left( 0,v;\frac{r}{ap}\right)\ll\\
\ll
\frac{(rp)^{\epsilon}}{k+T}\sum_{r/p\leq a < 2r/p}\frac{\sqrt{r}}{\sqrt{ap-r}}\ll\\\ll
\frac{(rp)^{\epsilon}}{k+T}\sqrt{r}\left(\frac{1}{\sqrt{p-r_0}} +\frac{\sqrt{r}}{p}\right)\ll
\frac{(rp)^{\epsilon}}{k+T}\sqrt{r}(1+\frac{\sqrt{r}}{p}).
\end{multline*}

The last summand $W_{p,3}(r)$ can be estimated using corollary \ref{boundI}
\begin{multline*}
W_{p,3}(r)\ll \sum_{a<r/p}|I(ap/r)|\ll \sum_{a<r/p}\frac{1}{\sqrt{1-ap/r}}\ll \sqrt{r}\left( 1+\frac{\sqrt{r}}{p}\right).
\end{multline*}

Finally,
\begin{equation*}
W_p(l)\ll V_p(l)+l^{\epsilon}\left(1+\sqrt{l}/p\right)\id_{l>p}.
\end{equation*}
Using \eqref{VN(l)bound} we obtain the assertion.
\end{proof}

\begin{lem}
One has
\begin{equation}
E_{p}^{3}(0,v, \lambda)\ll \frac{\log{pT}}{\sqrt{dp}}.
\end{equation}
\end{lem}
\begin{proof}
Consider
\begin{equation}
E_{p}^{3}(0,v, \lambda)\ll \frac{1}{\sqrt{dp}}(p^{2v-1}-1)\log{(1-2v)}.
\end{equation}
This gives
\begin{equation*}
E_{p}^{3}(0,v, \lambda)\ll
\begin{cases}
\frac{\log{T}}{\sqrt{dp}} \quad v \neq 0,\\
\frac{\log{p}}{\sqrt{dp}} \quad  v=0.
\end{cases}
\end{equation*}
\end{proof}

\begin{lem}\label{lem:e4}
One has
\begin{equation}
E_{p}^{4}(0,v,\lambda)\ll\frac{(dTk)^{\epsilon}}{\sqrt{d}}.
\end{equation}
\end{lem}
\begin{proof}
First, we estimate the second summand in \eqref{E4}. If $v \neq 0$, then it is bounded by $\log{T}/\sqrt{d}$.
Otherwise, letting $v \rightarrow 0$, the poles cancel out and the summand is dominated by $\log{kd}/\sqrt{d}$.
In total, this contributes as $\log{kdT}/\sqrt{d}$.
Now we proceed to estimate the first summand of \eqref{E4} containing the sum of Lerch zeta functions
\begin{multline*}
\sum_{\substack{b=1\\(b,qp)=1}}^{qp}\zeta(0,b/q,1+2v)=\sum_{n=1}^{\infty}\frac{1}{n^{1+2v}}
\sum_{\substack{b=1\\(b,qp)=1}}^{qp}\exp{\left( \frac{bpn}{qp}\right)}=\\
=\phi(qp)\sum_{n=1}^{\infty}\frac{\mu(q/(n,q))}{n^{1+2v}\phi(q/(n,q))}=\\=
\phi(qp)\sum_{m|q}\frac{\mu(m)}{\phi(m)}\frac{1}{(q/m)^{1+2v}}\sum_{\substack{n=1\\(n,m)=1}}^{\infty}\frac{1}{n^{1+2v}}.
\end{multline*}
Note that
\begin{equation*}
\sum_{\substack{n=1\\(n,m)=1}}^{\infty}
\frac{1}{n^{z}}=\sum_{n=1}^{\infty}\frac{1}{n^z}\sum_{a|(n,m)}\mu(a)=\zeta(z)\sum_{a|m}\frac{\mu(a)}{a^z}.
\end{equation*}
Thus
\begin{equation*}
\sum_{\substack{b=1\\(b,qp)=1}}^{qp}\zeta(0,b/q,1+2v)=\zeta(1+2v)\phi(qp)\sum_{m|q}\frac{\mu(m)}{\phi(m)}\frac{1}{(q/m)^{1+2v}}
\sum_{a|m}\frac{\mu(a)}{a^{1+2v}}.
\end{equation*}
Since for $q>1$
\begin{equation*}
\sum_{m|q}\frac{\mu(m)}{\phi(m)}\frac{m}{q}\sum_{a|m}\frac{\mu(a)}{a}=\sum_{m|q}\frac{\mu(m)}{q}=0,
\end{equation*}
there is no pole at $v=0$.
If $v=it \neq 0$ we use the bound $|\zeta(1+2v)|<\log{T}$ so that
\begin{equation*}
\sum_{\substack{b=1\\(b,qp)=1}}^{qp}\zeta(0,b/q,1+2v)\ll p(Tq)^{\epsilon}.
\end{equation*}
If $v=0$
\begin{equation*}
\sum_{\substack{b=1\\(b,qp)=1}}^{qp}\zeta(0,b/q,1)=1/2
\phi(qp)\sum_{m|q}\frac{\mu(m)}{\phi(m)}\frac{m}{q}\sum_{a|m}\frac{\mu(a)}{a}2 \log{\frac{m}{aq}},
\end{equation*}
and, therefore,
\begin{equation*}
\sum_{\substack{b=1\\(b,qp)=1}}^{qp}\zeta(0,b/q,1)\ll p(Tq)^{\epsilon}.
\end{equation*}
Combining all the results we have
\begin{equation*}
E_{p}^{4}(0,v,\lambda)\ll \frac{(dTk)^{\epsilon}}{\sqrt{d}}.
\end{equation*}

\end{proof}

\begin{lem}
One has
\begin{equation}
E_{p}^{2}(0,v, \lambda)\ll \frac{1}{\sqrt{dp}}(Tdp)^{\epsilon}.
\end{equation}
\end{lem}
\begin{proof}
Consider
\begin{equation*}
E_{p}^{2}(0,v, \lambda)\ll\frac{1}{p}\sum_{q|d}\frac{1}{2\pi \sqrt{d/p}}\Biggl| \sum_{\substack{b=1\\(b,qp)=1}}^{qp}\zeta(0, b/qp,1+2v)\Biggr|.
\end{equation*}
As in the proof of lemma \ref{lem:e4}
\begin{equation*}
 \sum_{\substack{b=1\\(b,qp)=1}}^{qp}\zeta(0, b/qp,1+2v)=
\zeta(1+2v)\phi(qp)\sum_{m|qp}\frac{\mu(m)}{\phi(m)}\left( \frac{m}{qp}\right)^{1+2v}\sum_{a|m}\frac{\mu(a)}{a^{1+2v}}.
\end{equation*}
The pole at $v=0$ cancels out for $qp>1$ and
\begin{equation*}
\sum_{\substack{b=1\\(b,qp)=1}}^{qp}\zeta(0, b/qp,1+2v)\ll (Tqp)^{\epsilon}.
\end{equation*}
The statement follows.
\end{proof}

\begin{cor}
The contribution of $E_{p}^{j}(0,v, \lambda),$  $j=2,3,4$ to \eqref{M2(p)} is bounded by
\begin{multline*}
\frac{2\pi i^{2k}}{p-p^{-1}}\sum_{d|l}\left( \frac{d}{l}\right)^{1/2-v}\times \\ \times \left( -\frac{E_{p}^2(0,v,\lambda)}{p^{1/2-v}}-
\frac{E^{3}_{p}(0,v,\lambda)}{p^{1/2+v}}+\frac{E_{p}^{4}(0,v,l)}{p}\right)\ll
\frac{(lpTk)^{\epsilon}}{p^2\sqrt{l}}.
\end{multline*}
\end{cor}

\section*{Funding}
The work of Dmitry Frolenkov   was supported by the Russian Science Foundation under grant [14-50-00005] and performed in Steklov Mathematical Institute of Russian Academy of Sciences.


\end{document}